\documentclass{icmart}
\usepackage{tikz}
\usepackage{mathrsfs}
\usepackage{hyperref}
\usepackage{mhequ}

\usetikzlibrary{shapes.misc}
\usetikzlibrary{shapes.symbols}
\usetikzlibrary{snakes}
\usetikzlibrary{decorations}

\colorlet{symbols}{black}
\tikzset{
	bdot/.style={circle,fill=black,draw=black,inner sep=0pt,minimum size=1pt},
	dot/.style={circle,fill=symbols,draw=symbols,inner sep=0pt,minimum size=0.4pt},
	xi/.style={circle,fill=symbols!10,draw=symbols,inner sep=0pt,minimum size=1.5pt},
	xi_old/.style={circle,fill=symbols!10,draw=symbols,inner sep=0pt,minimum size=1.2mm},
	ddd/.style={draw=black,dash pattern=on 0.4pt off 0.4pt, dash phase=0.4pt},
	Ip/.style={draw=symbols},
	I/.style={draw=symbols,
		decorate, decoration={zigzag,amplitude=0.4pt,segment length = 0.6pt,pre length=0.5pt,post length=0.5pt}},
	>=stealth,
	}
\makeatletter
\def\DeclareSymbol#1#2#3{\expandafter\gdef\csname MH@symb@#1\endcsname{\tikz[baseline=#2,scale=0.15,draw=symbols]{#3}}}
\def\<#1>{\csname MH@symb@#1\endcsname}
\makeatother

\DeclareSymbol{X}{-2.4}{\node[xi] {};}
\DeclareSymbol{1}{0}{\draw[white] (-.5,0) -- (.5,0); \draw[Ip] (0,0) node[dot] {} -- (0,1.25); \node[xi] at (0,1.4) {};}
\DeclareSymbol{d}{0}{\draw[white] (-.5,0) -- (.5,0); \draw[I] (0,0) node[dot] {} -- (0,1.25); \node[xi] at (0,1.4) {};}
\DeclareSymbol{2}{0}{\node(root)[dot] at (0,0) {}; \draw[Ip] (-0.5,1.4) node[xi] {} -- (root); \draw[Ip] (0.5,1.4) node[xi] {} -- (root);}

\DeclareSymbol{20}{-3}{\node(root)[dot] at (0,0) {}; \draw[Ip] (root) -- (0,-1) node[dot] {};\draw[Ip] (-.7,1) node[xi] {} -- (root); \draw[Ip](.7,1) node[xi] {} -- (root);}
\DeclareSymbol{2d}{-3}{\node(root)[dot] at (0,0) {}; \draw[I] (root) -- (0,-1) node[dot] {};\draw[Ip] (-.7,1) node[xi] {} -- (root); \draw[Ip] (.7,1) node[xi] {}--  (root);}
\DeclareSymbol{21}{-3}{\node(root)[dot] at (0,-1) {}; \draw[Ip] (-1,1) node[xi] {} -- (-0.5,0); \draw[Ip] (root) -- (-0.5,0); \draw[Ip] (root) -- (1,0) node[xi] {};\draw[Ip] (-0.5,0) node[dot] {} -- (0.5,1) node[xi] {};}

\DeclareSymbol{40}{-3}{\node(right)[dot] at (1.2,0) {};\node(left)[dot] at (-0.4,0) {};\node(root)[dot] at (0.4,-1) {};\draw[Ip] (-.8,1) node[xi] {} -- (left); \draw[Ip] (0,1) node[xi] {} -- (left);\draw[Ip] (.8,1) node[xi] {} -- (right); \draw[Ip] (right) -- (1.6,1) node[xi] {};\draw[Ip] (left) -- (root); \draw[Ip] (right) -- (root);}

\DeclareSymbol{211}{2}{\draw[Ip] (-1.2,2.4) node[xi] {} -- (-0.8,1.6)   node[dot] {} -- (-0.4,0.8)   node[dot] {} -- (0,0)  node[dot] {} -- (0.8,0.8) node[xi] {};\draw[Ip] (-0.4,0.8) -- (0.4,1.6) node[xi] {};\draw[Ip] (-0.8,1.6) -- (0,2.4) node[xi] {};}

\DeclareSymbol{210}{2}{\draw[Ip] (-1.2,2.4) node[xi] {} -- (-0.8,1.6)  node[dot] {} -- (-0.4,0.8)  node[dot] {} -- (0,0)  node[dot] {};\draw[Ip] (-0.4,0.8) -- (0.4,1.6) node[xi] {};\draw[Ip] (-0.8,1.6) -- (0,2.4) node[xi] {};}
\DeclareSymbol{21d}{2}{\draw[Ip] (-1.2,2.4) node[xi] {} -- (-0.8,1.6)  node[dot] {} -- (-0.4,0.8); \draw[I] (-0.4,0.8) node[dot] {} -- (0,0)  node[dot] {};\draw[Ip] (-0.4,0.8) -- (0.4,1.6) node[xi] {};\draw[Ip] (-0.8,1.6) -- (0,2.4) node[xi] {};}

\DeclareSymbol{10}{1}{\draw[Ip] (0,2) node[xi] {} -- (-1,1) node[dot] {} -- (0,0) node[dot] {};}
\DeclareSymbol{1d}{1}{\draw[Ip] (0,2) node[xi] {} -- (-1,1); \draw[I] (-1,1) node[dot] {} -- (0,0) node[dot] {};}
\DeclareSymbol{11}{1}{\draw[Ip] (0,2) node[xi] {} -- (-0.8,1) node[dot] {} -- (0,0) node[dot] {} -- (0.8,1) node[xi] {};}

\def\DD{\mathscr{D}}
\def\TT{\mathscr{T}}
\def\MM{\mathscr{M}}
\def\CC{\mathcal{C}}
\def\CD{\mathcal{D}}
\def\CB{\mathcal{B}}
\def\CG{\mathcal{G}}
\def\CH{\mathcal{H}}
\def\CS{\mathcal{S}}
\def\CR{\mathcal{R}}
\def\CI{\mathcal{I}}
\def\CF{\mathcal{F}}
\def\CU{\mathcal{U}}
\def\CV{\mathcal{V}}
\def\CJ{\mathcal{J}}
\def\CM{\mathcal{M}}
\def\CA{\mathcal{A}}
\def\CL{\mathcal{L}}
\def\CP{\mathcal{P}}

\def\CT{\mathcal{T}}

\def\E{{\mathbf E}}
\def\P{\mathbf{P}}
\def\R{\mathbf{R}}
\def\N{\mathbf{N}}
\def\one{\mathbf{1}}

\def\scal#1{\langle#1\rangle}

\def\PPi{\boldsymbol{\Pi}}

\def\eps{\varepsilon}
\def\div{\mathop{\rm div}}
\def\eqdef{:=}

\newtheorem{theorem}{Theorem}[section]

\newtheorem{lemma}[theorem]{Lemma}
\newtheorem{proposition}[theorem]{Proposition}

\theoremstyle{definition}
\newtheorem{definition}[theorem]{Definition}
\newtheorem{remark}[theorem]{Remark}

\title[Singular stochastic PDEs]{Singular stochastic PDEs}
\contact[M.Hairer@Warwick.ac.uk]{Mathematics Institute, The University of Warwick, U.K.}

\author[Martin Hairer]
{Martin Hairer\thanks{I am delighted to thank the Institute for Advanced Study 
for its warm hospitality and the `The Fund for Math' for funding my stay there.
This work was supported by the Leverhulme trust through a leadership award,
the Royal Society through a Wolfson research award, and the ERC through a consolidator award.
}}

\begin{document}

\begin{abstract}
We present a series of recent results on the well-posedness of very singular
parabolic stochastic partial differential equations. These equations are such that the question of
what it even means to be a solution is highly non-trivial. This problem can
be addressed within the framework of the recently developed theory of ``regularity structures'',
which allows to describe candidate solutions locally by a ``jet'', but where
the usual Taylor polynomials are replaced by a sequence of custom-built objects.
In order to illustrate the theory, we focus on the particular example of the Kardar-Parisi-Zhang 
equation, a popular model for interface propagation.
\end{abstract}



\maketitle\tableofcontents

\section{Introduction}

In this article, we report on a recently developed theory \cite{Regularity} allowing to give a robust meaning to
a large class of stochastic partial differential equations (SPDEs) that have traditionally been
considered to be ill-posed.
The general structure of these equations is
\begin{equ}[e:genEqu]
\CL u = F(u) + G(u)\xi\;,
\end{equ}
where the dominant linear operator $\CL$ is of parabolic (or possibly elliptic) type, $F$ and $G$
are local nonlinearities depending on $u$ and its derivatives of sufficiently low order, and $\xi$ 
is some driving noise. Problems arise when $\xi$ (and therefore also $u$) is so singular that 
some of the terms appearing in $F$ and / or the product between $G$ and $\xi$ are ill-posed.
For simplicity, we will consider all of our equations in a \textit{bounded} spatial region with
periodic boundary conditions.  

One relatively simple example of an ill-posed equation of the type \eqref{e:genEqu} 
is that of a system of equations with a nonlinearity of Burgers
type driven by space-time white noise:
\begin{equ}[e:Burgers]
\partial_t u = \partial_x^2 u + F(u)\,\partial_x u + \xi\;.
\end{equ}
(See Section~\ref{sec:prob} below for a definition of the space-time white noise $\xi$.)
Here, $u(x,t) \in \R^n$ and $F$ is a smooth matrix-valued function, so that one can in general
not rewrite the nonlinearity as a total derivative.
In this example, which was originally studied in \cite{Rough} but then further analysed in
the series of articles \cite{Jan,Hendrik,JanHendrik}, solutions at any fixed instant of time 
have exactly the same regularity (in space) as Brownian motion. As a consequence, $\partial_x u$ is
expected to ``look like''
white noise. It is of course very well-known from the study of ordinary stochastic differential equations (SDEs)
that in this case the product $F(u)\,\partial_x u$ is
``unstable'': one can get different answers depending on the type of limiting procedure used to define it.
This is the reason why one has different solution theories for SDEs: one obtains different answers, 
depending on whether they are interpreted in the It\^o or in
the Stratonovich sense \cite{Ito,Strat,WongZakai2,WongZakai}.

Another example is given by the KPZ equation \cite{KPZOrig} which can formally be written as
\begin{equ}[e:KPZ]
\partial_t h = \partial_x^2 h + (\partial_x h)^2 - C + \xi\;,
\end{equ}
and is a very popular model of one-dimensional interface propagation.
As in the case of \eqref{e:Burgers}, one expects solutions to this equation to ``look like''
Brownian motion (in space) for any fixed instant of time. Now the situation is much worse however:
the nonlinearity looks like the square of white noise, which really shouldn't make any sense!
In this particular case however, one can use a ``trick'', the Cole-Hopf transform, to reduce the problem to an
equation that has an interpretation within the framework of classical SPDE theory
\cite{MR1462228}. Furthermore, this ``Cole-Hopf solution'' was shown in \cite{MR1462228}
to be the physically relevant solution since it describes the mesoscopic fluctuations of
a certain microscopic interface growth model, see also \cite{Milton}.
On the other hand, the problem of interpreting these solutions
directly at the level of \eqref{e:KPZ} and to show their stability under suitable approximations 
had been open for a long time, before being addressed in \cite{KPZ}.

Both examples mentioned so far have only one space dimension. This particular feature
(together with some additional structure in the case of the KPZ equation, see Remark~\ref{rem:appl} below) allowed to treat them
by borrowing estimates and techniques from the theory of controlled rough paths \cite{Lyons,Max,Peter}.
This approach breaks down in higher spatial dimensions. 
More recently, a general theory of ``regularity structures'' was developed in \cite{Regularity}, 
which unifies many previous approaches and allows in particular to treat higher dimensional problems. 

Two nice examples of equations that can be treated with this new approach are given by
\minilab{e:SPDEs}
\begin{equs}
\partial_t \Phi &= \Delta \Phi + C \Phi - \Phi^3 + \xi\;, \label{e:AC} \\
\partial_t \Psi &= -\Delta\bigl(\Delta \Psi + C \Psi - \Psi^3\bigr) + \div \xi\;,  \label{e:CH}
\end{equs}
in space dimension $d=3$.
These equations can be interpreted as the natural ``Glauber'' and ``Kawasaki'' dynamics
associated to Euclidean $\Phi^4$ field theory in the context of stochastic quantisation \cite{ParisiWu}.
It is also expected to describe the dynamical mesoscale fluctuations for phase coexistence
models that are ``almost mean-field'', see \cite{MR1317994}. These equations cease to have function-valued 
solutions in dimension $d \ge 2$, so that the classical interpretation of the cubic nonlinearity loses its
meaning there. In two dimensions, a solution theory for these equations was developed in
\cite{AlbRock91}, which was later improved in \cite{MR1941997,MR2016604,MR2365646}, 
see Section~\ref{sec:DPD} below. The case $d=3$ (which is the physically relevant one in the 
interpretation as dynamical fluctuations for phase coexistence models) had remained open
and was eventually addressed in \cite{Regularity}.

A final example of the kind of equations that can be addressed by the theory exposed in these notes
(but this list is of course not exhaustive) is a continuous analogue to the classical parabolic Anderson model \cite{MR1185878}:
\begin{equ}[e:PAM]
\partial_t u = \Delta u + u\,\eta + C u\;,
\end{equ}
in dimensions $d \in \{2,3\}$.
In this equation, $\eta$ denotes a noise term that is white in space, but constant in time.
This time, the problem is that in dimension $d \ge 2$, the product $u\,\eta$ ceases to make
sense classically, as a consequence of the lack of regularity of $u$. 

The following ``meta-theorem'' (formulated in a somewhat vague sense, precise formulations differ
slightly from problem to problem and can be found in the abovementioned articles) shows 
in which sense one can give meaning to all of these equations.

\begin{theorem}\label{theo:main}
Consider the sequence of classical solutions to 
any of the equations \eqref{e:Burgers}--\eqref{e:PAM}
with $\xi$ (resp.\ $\eta$) replaced by a smooth regularised noise $\xi_\eps$ and $C = C_\eps$ depending on $\eps$.
Then, there exists s choice $C_\eps \to \infty$ such that this sequence of solutions converges to
a limit in probability, locally in time. Furthermore, this limit is universal, i.e.\ does not depend
on the details of the regularisation $\xi_\eps$.
\end{theorem}

Besides these convergence results, the important fact here is that the limit is \textit{independent} of the
precise details of the regularisation mechanism.
In addition, the theory of regularity structures also yields rates 
of convergence, as well as
an intrinsic description of these limits. It also provides automatically a very detailed local
description of these limits. 

The starting point for the theory is the remark that the usual way
of determining the ``smoothness'' of a function is to measure how well it can locally be approximated by a 
polynomial. In the case of singular SPDEs, polynomials are a terrible way of approximating the solution.
Instead, we can fare much better if we ``custom-build'' a collection of functions / distributions 
out of the driving noise and approximate solutions locally by finite linear combinations of these
objects. Various renormalisation procedures can then be built into the ``multiplication table'' for 
these objects, which translates into diverging correction terms at the level of the original equations.

The aim of this article is to give an overview of the ingredients involved in the proof of a result
like Theorem~\ref{theo:main}. We structure this as follows. In Section~\ref{sec:Holder},
we recall a number of properties and definitions of H\"older spaces of positive (and negative!) order
that will be useful for our argument. In Section~\ref{sec:General}, we then explain how, using only standard tools,
it is possible to provide a robust solution theory for not-so-singular SPDEs, like for example
\eqref{e:SPDEs} in dimension $d=2$. Section~\ref{sec:regular} is devoted to a short overview of the main definitions
and concepts of the abstract theory of regularity structures which is a completely general way of formalising
the properties of objects that behave ``like Taylor polynomials''. Section~\ref{sec:SPDE} then finally shows how
one can apply this general theory to the specific context of the type of parabolic SPDEs considered above, 
how renormalisation procedures can be built into the theory,
and how this affects the equations.

Throughout the whole article, our argumentation will remain mostly at the heuristic level, but we will 
make the statements and definitions as precise as possible.

\subsection{An alternative approach}\label{sec:Bony}

A different approach to building solution theories for singular PDEs
was developed simultaneously to the one presented here by Gubinelli \& Al in \cite{PAMPreprint}.
That approach is based on the properties of Bony's paraproduct \cite{Bony,MR2682821,BookChemin}, in 
particular on the paralinearisation formula. 
It is also
able to deal with the KPZ equation (at least in principle) or the dynamical
$\Phi^4_3$ model \eqref{e:AC}, see \cite{Chouk}.

One advantage
is that in the paraproduct-based approach one generally deals with globally defined objects rather than 
the ``jets'' used in the theory of regularity structures. It also uses some already well-studied
objects, so that it can build on a substantial body of existing literature.
However, this comes at the expense of achieving a less clean break between the analytical and the algebraic aspects 
of a given problem and obtaining less detailed information about the solutions. 
Furthermore, its scope is not as wide as that of the theory of regularity structures.
For example, it cannot deal with classical one-dimensional parabolic SPDEs of the type 
\eqref{e:SPDE} below, for which we are able to obtain a result of Wong-Zakai type \cite{HPP}.
See also Remark~\ref{rem:appl} below for more details.

\section{Some properties of H\"older spaces}
\label{sec:Holder}

We recall in this section a few standard results from harmonic analysis that are
very useful to have in mind. Note first that the linear part of all of the equations
described in the introduction is invariant under some space-time scaling. In the case
of the heat equation, this is the parabolic scaling. In other words, if $u$ is
a solution to the heat equation, then $\tilde u(t,x) = u(\lambda^{-2} t, \lambda^{-1} x)$ is
also a solution to the heat equation. (Leaving aside the fact that we really consider 
solutions defined only on a fixed bounded domain.)

This suggests that we should look for solutions in function / distribution spaces 
respecting this scaling. Given a smooth compactly supported test function $\varphi$
and a space-time coordinate $z = (t,x)$, we
 henceforth denote by $\varphi_z^\lambda$ the test function
\begin{equ}
\varphi_z^\lambda(s,y) = \lambda^{-d-2} \varphi\bigl(\lambda^{-2}(s-t), \lambda^{-1}(y-x)\bigr)\;,
\end{equ}
where $d$ denotes the spatial dimension and the factor $\lambda^{-d-2}$ is chosen so that
the integral of $\varphi_z^\lambda$ is the same as that of $\varphi$. In the case of the 
stochastic Cahn-Hilliard equation \eqref{e:CH}, we would naturally use instead a temporal scaling
of $\lambda^{-4}$ and the prefactor would then be $\lambda^{-d-4}$.

With these notations at hand, we define spaces of distributions $\CC^\alpha$ for $\alpha < 0$
in the following way. Denoting by $\CB_\alpha$ the set of smooth test functions
$\varphi \colon \R^{d+1} \to \R$ that are supported in the centred ball of radius $1$ and such that 
their derivatives of order up to $1+|\alpha|$ are uniformly bounded by $1$, we set

\begin{definition}\label{def:Calpha}
Let $\eta$ be a distribution on $d+1$-dimensional space-time and let $\alpha < 0$. 
We say that $\eta \in \CC^\alpha$ if the bound
\begin{equ}
\bigl|\eta(\varphi_z^\lambda)\bigr| \lesssim \lambda^\alpha\;,
\end{equ}
holds uniformly over all $\lambda \in (0,1]$, all $\varphi \in \CB_\alpha$, and locally
uniformly over $z \in \R^{d+1}$.
\end{definition}

For $\alpha \ge 0$, we say that a function $f\colon \R^{d+1} \to \R$ belongs to $\CC^\alpha$
if, for every $z \in \R^{d+1}$ there exists a polynomial $P_z$ of (parabolic) degree at most
$\alpha$ and such that the bound
\begin{equ}
|f(z') - P_z(z')| \lesssim |z-z'|^\alpha\;,
\end{equ}
holds locally uniformly over $z$ and uniformly over all $z'$ with $|z'-z| \le 1$. 
Here, we say that a polynomial $P$ in $z = (t,x)$ is of parabolic degree $n$ if each monomial is of
degree $k$ in $t$ and $m$ in $x$ for some $k$ and $m$ with $m + 2k \le n$. In other words, the degree of the 
time variable ``counts double''. For $z = (t,x)$,
we furthermore write $|z| = |t|^{1/2} + |x|$. (When treating \eqref{e:CH}, 
powers of $t$ count four times and one writes $|z| = |t|^{1/4} + |x|$.)

\begin{remark}
Note that in both cases, the definition of $\CC^\alpha$ is local, i.e.\ we do not impose
anything concerning the behaviour at infinity. It is of course straightforward to modify the 
definitions in order to impose uniform control in space, possibly based on a weight function with
suitable properties. In our setting, we will always be in a situation where our objects are 
periodic in space and we only care about bounds over some compact interval in time, so this is not an issue.
\end{remark}

We now collect a few important properties of the spaces $\CC^\alpha$.

\subsection{Analytical properties}

First, given a function and a distribution (or two distributions) it is natural to ask under what
regularity assumptions one can give an unambiguous meaning to their product. 
It is well-known, at least in the Euclidean case but the extension to the parabolic case is 
straightforward, that the following result yields a sharp criterion for when, in the absence of
any other structural knowledge, one can multiply a function and distribution of prescribed 
regularity \cite[Thm~2.52]{BookChemin}.

\begin{theorem}\label{theo:mult}
Let $\alpha, \beta \neq 0$. Then, the map $(f, g)\mapsto f\cdot g$ defined on all pairs of continuous
functions extends to a continuous bilinear
map from $\CC^\alpha \times \CC^\beta$ to the space of all distributions if and only if $\alpha + \beta > 0$.
Furthermore, if $\alpha + \beta > 0$, the image of the multiplication operator is $\CC^{\alpha \wedge \beta}$.
\end{theorem}

Another important property of these spaces if given by how they transform under convolution
with singular kernels. Let $K \colon \R^{d+1}$ be a function that is smooth away from the origin
and supported in the centred ball of radius $1$. One should think of $K$ as being a truncation of
the heat kernel $\CG$ in the sense that $\CG = K + R$ where $R$ is a smooth space-time function.
We then say that $K$ is of order $\beta$ (in the case of a truncation of the heat kernel one
has $\beta = 2$) if one can write $K = \sum_{n \ge 0} K_n$ for kernels $K_n$ which are
supported in the centred ball of radius $2^{-n}$ and such that
\begin{equ}[e:propKn]
\sup_z |D^k K_n(z)| \lesssim 2^{((d+2) + |k| - \beta)n}\;,
\end{equ}
for any fixed multiindex $k$, uniformly in $n$. Here, we set $|k| = 2k_0 + \sum_{i \neq 0} k_i$,
i.e.\ ``time counts double'' as before. Multiplying the heat kernel with a suitable partition of
the identity, it is straightforward to verify that this bound is indeed satisfied. It is also
satisfied with $\beta = 4$ for the Green's function of $\partial_t + \Delta^2$, provided
that the parabolic scaling is adjusted accordingly.

With these notations at hand, one has the following very general Schauder estimate,
see for example \cite{Schauder,MR1459795} for special cases.

\begin{theorem}\label{theo:schauder}
Let $\beta > 0$, let $K$ be a kernel of order $\beta$, and let $\alpha \in \R$ be such that
$\alpha + \beta \not \in \N$. Then, the convolution operator $\eta \mapsto K \star \eta$ is continuous from
$\CC^\alpha$ into $\CC^{\alpha + \beta}$. 
\end{theorem}

\begin{remark}
The condition $\alpha + \beta \not \in \N$ seems somewhat artificial. It can be dispensed with 
by slightly changing the definition of $\CC^\alpha$: define $\tilde \CC^\alpha$ for $\alpha \in \R$
by copying Definition~\ref{def:Calpha},
but with $\CB_\alpha$ replaced by $\tilde \CB_\alpha$. This consists of those test functions
$\varphi \in \CB_\alpha$ which furthermore satisfy $\int \varphi(z)\,P(z)\,dz = 0$ for every
polynomial $P$ of parabolic degree at most $\alpha$. (For $\alpha < 0$, one then obtains 
$\tilde \CC^\alpha = \CC^\alpha$ since this condition is empty.)
The spaces $\tilde \CC^\alpha$ actually coincide with the spaces 
$\CC^\alpha$ defined above for $\alpha \in \R \setminus \N$, but
they yield different spaces (the Zygmund spaces) for integer values of $\alpha$.
It turns out that Theorem~\ref{theo:schauder} holds without any restriction on $\alpha$ if
one replaces $\CC^\alpha$ and $\CC^{\alpha+\beta}$ by $\tilde \CC^\alpha$ and 
$\tilde \CC^{\alpha+\beta}$ in the statement.
\end{remark}

\subsection{Probabilistic properties}
\label{sec:prob}

Let now $\eta$ be a random distribution, which we define in general as a continuous linear map $\varphi \mapsto \eta(\varphi)$ from 
the space of compactly supported smooth test functions into the space of square integrable random
variables on some fixed probability space $(\Omega, \P)$. We say that it satisfies \textit{equivalence
of moments} if, for every $p \ge 1$ there exists a constant $C_p$ such that the bound
\begin{equ}
\E |\eta(\varphi)|^{2p} \le C_p \bigl(\E |\eta(\varphi)|^{2}\bigr)^p\;,
\end{equ} 
holds for uniformly over all test functions $\varphi$. This is of course the case if the random variables
$\eta(\varphi)$ are Gaussian, but it also holds if they take values in an inhomogeneous
Wiener chaos of fixed order \cite{Nualart}. 

Given a stationary random distribution $\eta$ and a (deterministic) distribution $C$, 
we say that $\eta$ has covariance $C$ if one has the identity 
$\E \eta(\varphi)\eta(\psi) = \scal{C \star \varphi,\psi}$, where $\scal{\cdot,\cdot}$ denotes the
$L^2$-scalar product. More informally, this states that 
\begin{equ}
\E \eta(\varphi)\eta(\psi) = \int \varphi(z)\psi(z') C(z-z')\,dz\,dz'\;.
\end{equ}
With these notations at hand, space-time white noise is the Gaussian random distribution
on $\R^{d+1}$ with covariance given by the delta distribution. (There is of course also a natural
periodic version of space-time white noise.) In other words, space-time white noise $\xi$ is a random 
distribution such that $\xi(\varphi)$ is centred Gaussian for every test function $\varphi$ and such that
furthermore $\E \xi(\varphi)\xi(\psi) = \scal{\varphi,\psi}_{L^2}$.

Similarly to the case of stochastic processes, a random distribution $\tilde \eta$ is said to be
a \textit{version} of $\eta$ if, for every fixed test function $\varphi$, the identity $\tilde \eta(\varphi) = \eta(\varphi)$ holds almost surely. One then has the following Kolmogorov-type criterion, a proof
of which can be found for example in \cite{Regularity}.

\begin{theorem}\label{theo:scaling}
Let $\eta$ be a stationary random distribution satisfying equivalence of moments and such that,
for some $\alpha < 0$, the bound
\begin{equ}
\E |\eta(\varphi_z^\lambda)|^2 \lesssim \lambda^{2\alpha}\;,
\end{equ}
holds uniformly over $\lambda \in (0,1]$ and $\varphi \in \CB_\alpha$. Then, for any $\kappa > 0$, 
there exists a $\CC^{\alpha - \kappa}$-valued random variable $\tilde \eta$ which is a version of $\eta$.
\end{theorem}

From now on, we will make the usual abuse of terminology and not distinguish between different 
versions of a random distribution.

\begin{remark}\label{rem:STWN}
It follows immediately from the scaling properties of the $L^2$ norm that
one can realise space-time white noise as a random variable in $\CC^{-{d\over 2} - 1 - \kappa}$
for every $\kappa > 0$. This is sharp in the sense that it can \textit{not} be realised
as a random variable in $\CC^{-{d\over 2} - 1}$. This is akin to the fact that Brownian motion
has sample paths belonging to $\CC^\alpha$ for every $\alpha < {1\over 2}$, but \textit{not} for $\alpha = {1\over 2}$.
\end{remark}

Let now $K$ be a kernel of order $\beta$ as before, let $\xi$ be space-time white noise,
and set $\eta = K \star \xi$. It then follows from either Theorem~\ref{theo:scaling} directly,
or from Theorem~\ref{theo:schauder} combined with Remark~\ref{rem:STWN}, that 
$\eta$ belongs almost surely to $\CC^\alpha$ for every $\alpha < \beta -{d\over 2} - 1$.
We now turn to the question of how to define powers of $\eta$. If $\beta \le {d\over 2} + 1$,
$\eta$ is not a random function, so that its powers are in general undefined.

Recall that if $\xi$ is space-time white noise and $L^2(\xi)$ denotes the space of square-integrable
random variables
that are measurable with respect to the $\sigma$-algebra generated by $\xi$, then $L^2(\xi)$
can be decomposed into a direct sum $L^2(\xi) = \bigoplus_{m \ge 0} \CH^m(\xi)$
so that $\CH^0$ contains constants, 
$\CH^1$ contains random variables of the form $\xi(\varphi)$ with $\varphi \in L^2$, and
$\CH^m$ contains suitable generalised Hermite polynomials of order $m$ in the elements of $\CH^1$,
see \cite{Malliavin,Nualart} for details. Elements of $\CH^m$ have a representation
by square-integrable kernels of $m$ variables, and this representation is unique if
we impose that the kernel is symmetric under permutation of its arguments.
In other words, one has a surjection $I^{(m)} \colon L^2(\R^{d+1})^{\otimes m} \to \CH^m$
and $I^{(m)}(L) = I^{(m)}(L')$ if and only if the symmetrisations of $L$ and $L'$ coincide. 

In the particular case where $K$ is non-singular, $\eta$ is a random function and 
its $n$th power $\eta^n$ can be represented as
\begin{equs}
\eta^n(z) &= \sum_{2m < n} P_{m,n} C^m\, I^{(n-2m)}(K_z^{(n-2m)})\;,\\
K_z^{(r)}(z_1,\ldots,z_r) &\eqdef K(z-z_1)\cdots K(z-z_r)\;,
\end{equs}
for some combinatorial factors $P_{m,n}$,
where we have set $C = \int K^2(z)\,dz$. Let now $K_\eps$ denotes some smooth approximation to
a kernel $K$ that is not square integrable. There are then two sources of divergencies in this 
representation as $\eps \to 0$. First, we see that the sequence of constants $C_\eps = \int K_\eps^2(z)\,dz$
will necessarily diverge as $\eps \to 0$. Second, we see that the kernels $K_{\eps;z}^{(r)}$ defined
as above do not converge to a square-integrable limit as $\eps \to 0$.

While the first problem is serious, the second problem is a spurious consequence of the 
fact that we insist on evaluating $\eta$ at a fixed space-time point $z$. If instead we consider
$\eta^n(\varphi)$ for a smooth compactly supported test function $\varphi$, the above representation 
becomes
\begin{equ}
\eta^n(\varphi) = \sum_{2m < n} P_{m,n} C^m\, I^{(n-2m)}(K_\varphi^{(n-2m)})\;,
\end{equ}
where
\begin{equ}
K_\varphi^{(r)}(z_1,\ldots,z_r) \eqdef \int K(z-z_1)\cdots K(z-z_r) \,\varphi(z)\,dz\;.
\end{equ}
A simple calculation then shows that

\begin{proposition}\label{prop:Wick}
If $K$ is compactly supported, then $K_\varphi^{(n)}$ is square integrable if the function
$(K \star \hat K)^n$, where $\hat K(z) = K(-z)$, is integrable.
\end{proposition}

We now \textit{define} the $n$th Wick power $\eta^{\diamond n}$ of $\eta$ as the random
distribution given by
\begin{equ}
\eta^{\diamond n}(\varphi) = I^{(n)}(K_\varphi^{(n)})\;,
\end{equ}
which, by Proposition~\ref{prop:Wick}, 
makes sense as soon as $K \star \hat K \in L^n(\R^{d+1})$. One then has the following result,
a version of which can be found for example in \cite{Dobrushin}.

\begin{proposition}\label{prop:Wick}
Let $K$ be a compactly supported kernel of order $\beta \in ({d+2\over 2}(1-{1\over n}), {d+2\over 2})$ and let
$\eta = K \star \xi$ as above. Then, $\eta^{\diamond n}$ is well-defined and belongs
almost surely to $\CC^\alpha$ for every $\alpha < (2\beta - d - 2){n\over 2}$. 
\end{proposition}

\begin{proof}
A simple calculation shows that 
\begin{equ}
\bigl| \bigl(K \star \hat K\bigr)(z)\bigr|^n \lesssim |z|^{(2\beta - d - 2)n}\;,
\end{equ}
so that $\|K_{\varphi_z^\lambda}^{(n)}\|_{L^2}^2 \lesssim \lambda^{(2\beta - d - 2)n}$. The
claim then follows from Theorem~\ref{theo:scaling}, noting that random variables belonging to
a Wiener-It\^o chaos of finite order satisfy the equivalence of moments. 
\end{proof}

It is important to note that this result is stable: replacing $K$ by a smoothened
kernel $K_\eps$ and letting $\eps \to 0$ yields convergence in probability of
$\eta_\eps^{\diamond n}$ to $\eta^{\diamond n}$ in $\CC^\alpha$ (with $\alpha$
as in the statement of the proposition) for most ``reasonable'' choices of $K_\eps$.
Furthermore, for fixed $\eps > 0$, one has an explicit formula relating $\eta_\eps^{\diamond n}$
to $\eta_\eps$:
\begin{equ}[e:propetaeps]
\eta_\eps^{\diamond n}(z) = H_n(\eta_\eps(z), C_\eps)\;,
\end{equ}
where the rescaled Hermite polynomials $H_n(\cdot,C)$ are related to the standard Hermite
polynomials by $H_n(u,C) = C^{n/2} H_n(C^{-1/2} u)$ and we have set $C_\eps = \int K_\eps^2(z)\,dz$.

\section{General methodology}
\label{sec:General}

The general methodology for providing a robust meaning to equations
of the type presented in the introduction is as follows. We remark that the main reason why
these equations seem to be ill-posed is that there is no canonical way of multiplying
arbitrary distributions. The distributions appearing in our setting are however
not arbitrary. For instance, one would expect solutions to semilinear equations of 
this type to locally ``look like'' the solutions to the corresponding 
linear problems. This is because, unlike hyperbolic or dispersive equations, 
parabolic (or elliptic) equations to not transport singularities.
This gives hope that if one could somehow make sense of the 
nonlinearity, when applied to the solution to the linearised equation
(which is a Gaussian process and therefore amenable to explicit calculations), then
one could maybe give meaning to the equations themselves.

\subsection{The Da Prato-Debussche trick}
\label{sec:DPD}

In some situations, one can apply this idea directly, and this was originally 
exploited in the series of articles \cite{MR1941997,MR2016604,MR2365646}. Let us focus on the example of 
the dynamical $\Phi^4$ model in dimension $2$, which is formally given by
\begin{equ}
\partial_t \Phi = \Delta \Phi +  C \Phi - \Phi^3 + \xi\;,
\end{equ}
where $\xi$ is (spatially periodic) space-time white noise in space dimension $2$.

Let now $\xi_\eps$ denote a smoothened version of $\xi$ given for example
by $\xi_\eps = \rho_\eps \star \xi$, where $\rho_\eps(t,x) = \eps^{-4} \rho(\eps^{-2}t, \eps^{-1}x)$,
for some smooth compactly supported space-time mollifier $\rho$. In this case, denoting again
by $K$ a cut-off version of the heat kernel and noting that $K$ is of order $2$ 
(and therefore also of every order less than $2$), it is immediate that $\eta = K \star \xi$ satisfies
the assumptions of Proposition~\ref{prop:Wick} for every integer $n$.

In view of \eqref{e:propetaeps}, this suggests that it might be possible to show that 
the solutions to
\begin{equs}[e:Phi42reg]
\partial_t \Phi_\eps &= \Delta \Phi_\eps + 3C_\eps \Phi_\eps - \Phi_\eps^3 + \xi_\eps \\
&= \Delta \Phi_\eps - H_3(\Phi_\eps, C_\eps) + \xi_\eps\;,
\end{equs}
with $C_\eps = \int K_\eps^2(z)\,dz$ as above, where $K_\eps = \rho_\eps \star K$, converge to a distributional
limit as $\eps \to 0$. This is indeed the case, and the argument goes as follows. 
Writing $\eta_\eps = K_\eps \star \xi$ and $v_\eps = \Phi_\eps - \eta_\eps$ with $\Phi_\eps$ the solution
to \eqref{e:Phi42reg}, we deduce that $v_\eps$ solves the equation
\begin{equ}
\partial_t v_\eps = \Delta v_\eps - H_3(\eta_\eps + v_\eps, C_\eps) + R_\eps\;,
\end{equ}
for some smooth function $R_\eps$ that converges to a smooth limit $R$ as $\eps \to 0$.
We then use elementary properties of Hermite polynomials to rewrite this as
\begin{equs}
\partial_t v_\eps &= \Delta v_\eps - \bigl(H_3(\eta_\eps, C_\eps) + 3 v_\eps H_2(\eta_\eps, C_\eps) + 3 v_\eps^2 \,\eta_\eps + v_\eps^3\bigr) + R_\eps \\
&= \Delta v_\eps - \bigl(\eta_\eps^{\diamond 3} + 3 v_\eps \eta_\eps^{\diamond 2} + 3 v_\eps^2 \,\eta_\eps + v_\eps^3\bigr) + R_\eps\;.
\end{equs}
By Proposition~\ref{prop:Wick} (and the remarks that follow), we see that $\eta_\eps^{\diamond n}$ 
converges in probability to a limit $\eta^{\diamond n}$ in every space $\CC^\alpha$ for $\alpha < 0$.
We can then \textit{define} a random distribution $\Phi$ by $\Phi = \eta + v$, where $v$ is the solution to
\begin{equ}[e:v]
\partial_t v = \Delta v - \bigl(\eta^{\diamond 3} + 3 v \eta^{\diamond 2} + 3 v^2 \,\eta + v^3\bigr) + R\;.
\end{equ}
As a consequence of Theorem~\ref{theo:schauder} (combined with additional estimates showing that the $\CC^\gamma$-norm
of $K \star (f \one_{t > 0})$ is small over short times provided that $f \in \CC^{\alpha}$ for $\alpha \in (-2,0)$ 
and $\gamma < \alpha + \beta$), it is relatively easy to show that \eqref{e:v} has local solutions, and that these
solutions are robust with respect to approximations of $\eta^{\diamond n}$ in $\CC^\alpha$ for $\alpha$ sufficiently close
to $0$. In particular, this shows that one has $\Phi_\eps \to \Phi$ in probability, at least locally in time for short times.

\begin{remark}
The dynamical $\Phi^4$ model in dimension $2$ was previously 
constructed in \cite{AlbRock91} (see also the earlier work \cite{MR815192} where a related
but different process was constructed), but that construction relied heavily
on \textit{a priori} knowledge about its invariant measure and it was not clear how robust the construction
was with respect to perturbations. For example, the construction we just sketched gives solutions for 
every (and not just ``almost every'') initial condition in 
$\CC^\eta$ for not too small $\eta$. (It turns out that $\eta > -2/3$ is sufficient and that this is sharp.)
\end{remark}

\begin{remark}
The regularisation \eqref{e:Phi42reg} is not always the most natural one because there is no closed-form expression
for the invariant measure of \eqref{e:Phi42reg} in general. In this sense, a ``better'' regularisation would be
given for example by solutions to $\partial_t \Phi_\eps = (\Delta - \eps^2 \Delta^2) \Phi_\eps + 3C_\eps \Phi_\eps - \Phi_\eps^3 + \xi$. It is equally straightforward to show that this also converges to the same limiting process.
\end{remark}

\subsection{Breakdown of the argument and a strategy to rescue it}

While the argument outlined above works very well for a number of equations, it unfortunately breaks down
for the equations mentioned in the introduction. Indeed, consider again \eqref{e:AC}, but this 
time in space dimension $d = 3$. In this case, one has $\eta \in \CC^{-{1\over 2}-\kappa}$ for every $\kappa > 0$
and, by Proposition~\ref{prop:Wick}, one can still make sense of $\eta^{\diamond n}$ for $n < 5$. 
One could therefore hope to define again a solution $\Phi$ by setting $\Phi = \eta + v$ with $v$ the solution to \eqref{e:v}.
Unfortunately, this is doomed to failure: since $\eta^{\diamond 3} \in \CC^{-{3\over 2} - \kappa}$ (but no better),
one can at best hope to have $v \in \CC^{{1\over 2}-\kappa}$. As a consequence, both products 
$v \cdot \eta^{\diamond 2}$ and $v^2 \cdot \eta$ fall outside of the scope of Theorem~\ref{theo:mult}
and we cannot make sense of \eqref{e:v}.

One might hope at this stage that the Da Prato-Debussche trick could be iterated to improve things: identify the
``worst'' term in the right hand side of \eqref{e:v}, make sense of it ``by hand'', and try to obtain a well-posed
equation for the remainder. While this strategy can indeed be fruitful and allows us to deal with slightly more singular
problems, it turns out to fail in this situation. Indeed, no matter how many times we iterate this trick, the 
right hand side of the equation for the remainder $v$ will \textit{always} contain a term proportional to 
$v \cdot \eta^{\diamond 2}$. As a consequence, one can \textit{never} hope to obtain a remainder of regularity better
than $\CC^{1-\kappa}$ which, since $\eta^{\diamond 2} \in \CC^{-1-\kappa}$, shows that it is not possible to obtain
a well-posed equation by this method. See also Remark~\ref{rem:appl} below for a more systematic explanation of when 
this trick fails.

In some cases, one does not even know how to get started: consider the class of ``classical'' one-dimensional
stochastic PDEs given by
\begin{equ}[e:SPDE]
\partial_t u = \partial_x^2 u + f(u) + g(u)\xi\;,
\end{equ}
where $\xi$ denotes space-time white noise, $f$ and $g$ are fixed smooth functions from $\R$ to $\R$,
and the spatial variable $x$ takes values on the circle. Then, we know in principle how to use It\^o calculus to make
sense of \eqref{e:SPDE} by rewriting it as an integral equation and interpreting the integral against $\xi$
as an It\^o integral, see \cite{DPZ}. However, this notion of solution is not very robust under approximations
since space-time regularisations of the driving noise $\xi$ typically destroy the probabilistic structure 
required for It\^o integration.
This is in contrast to the solution theory sketched in Section~\ref{sec:DPD}
which was very stable under approximations of the driving noise, even though it required suitable
adjustments to the equation itself. Unfortunately, the argument of Section~\ref{sec:DPD} (try to find some
function / distribution $\eta$ so that $v = u - \eta$ has better regularity properties and then obtain
a well-posed equation for $v$) appears to break down completely. 

The main idea now is that even though we may not be able to find a global object $\eta$ so that $u-\eta$
has better regularity, it might be possible to find a \textit{local} object that does the trick
at any one point. More precisely, setting again $\eta = K \star \xi$ for a truncation $K$ of the
heat kernel (this time $\eta$ is a H\"older continuous function in $\CC^{{1\over 2}-\kappa}$ for every $\kappa > 0$ by
Theorems~\ref{theo:schauder} and \ref{theo:scaling}), one would expect solutions to \eqref{e:SPDE}
to be well approximated by
\begin{equ}[e:approxu]
u(z') \approx u(z) + g(u(z)) \bigl(\eta(z') - \eta(z)\bigr)\;.
\end{equ}
Why is this the case? The intuition is that since $K$ is regular everywhere except at the origin, 
convolution with $K$ is ``almost'' a local operator, modulo more regular parts. Since, near any fixed point $z$,
we would expect $g(u)\xi$ to ``look like'' $g(u(z))\xi$ this suggests that near that point $z$, 
the function $K \star (g(u)\xi)$ should ``look like'' $g(u(z)) \eta$, which is what \eqref{e:approxu} formalises.

Note that this looks very much like a first-order Taylor expansion, but with $\eta(z') - \eta(z)$ playing 
the role of the linear part $z'-z$. If we assume that \eqref{e:approxu} yields a good approximation to $u$, then
one would also expect that
\begin{equ}
g(u(z')) \approx g(u(z)) + g'(u(z)) g(u(z)) \bigl(\eta(z') - \eta(z)\bigr)\;,
\end{equ}
so that $g(u)$ has again a ``first-order Taylor expansion'' of the same type as the one for $u$.
One could then hope that if we know somehow how to multiply $\eta$ with $\xi$,
this knowledge could be leveraged to define the product between $g(u)$ and $\xi$ in a robust way.
It turns out that this is \textit{not} quite enough for the situation considered here. However, this general strategy turns out to 
be very fruitful, provided that we also control higher-order local expansions of $u$, and this is precisely
what the theory of regularity structures formalises \cite{Regularity,HPP}. 
In particular, besides being applicable to \eqref{e:SPDE}, it also applies to  all of the equations mentioned in the introduction.

\section{Regularity structures}
\label{sec:regular}

We now describe a very general framework in which one can formulate ``Taylor expansions'' 
of the type \eqref{e:approxu}. We would like to formalise the following features of Taylor expansions.
First, the coefficients of a Taylor expansion (i.e.\ the value and derivatives of a given function
in the classical case or the coefficients $u(z)$ and $g(u(z))$ in the case \eqref{e:approxu}) correspond
to terms of different degree / homogeneity and should therefore naturally be thought of as elements in
some graded vector space. Second, an expansion around a given point can be reexpanded around a different point
at the expense of changing coefficients, like so:
\begin{equs}
a\cdot 1 + b\cdot x + c\cdot  x^2 &= \bigl(a + bh + ch^2\bigr)\cdot 1 + \bigl(b + 2ch\bigr)\cdot (x-h) +c\cdot (x-h)^2\;,\\
u\cdot 1 + g(u) \cdot \bigl(\eta(z') - \eta(z)\bigr) & = 
\bigl(u + g(u)(\eta(z'') - \eta(z))\bigr)\cdot 1 + g(u) \cdot \bigl(\eta(z') - \eta(z'')\bigr)\;.
\end{equs}
Lastly, we see from these expressions that if we order coefficients by increasing homogeneity,
then the linear transformation performing the reexpansion has 
an upper triangular structure with the identity on the diagonal. 

\subsection{Basic definitions}\label{sec:def}

The properties just discussed are reflected in the following algebraic structure.

\begin{definition}\label{def:regStruct}
A \textit{regularity structure} $\TT = (A, T, G)$ consists of the following elements:
\begin{itemize}
\item An index set $A \subset \R$ such that $0 \in A$, $A$ is bounded from below, and $A$ is locally finite.
\item A \textit{model space} $T$, which is a graded vector space $T = \bigoplus_{\alpha \in A} T_\alpha$,
with each $T_\alpha$ a Banach space; elements in $T_\alpha$ are said to have {\it homogeneity} (or {\it degree}) 
$\alpha$. Furthermore $T_0$ is one-dimensional and has a distinguished basis vector $\one$.
Given $\tau \in T$, we write $\|\tau\|_\alpha$ for the norm of its component in $T_\alpha$.
\item A \textit{structure group} $G$ of (continuous) linear operators acting on $T$ such that, for every $\Gamma \in G$, every $\alpha \in A$,
and every ${\tau_\alpha} \in T_\alpha$, one has
\begin{equ}[e:coundGroup]
\Gamma {\tau_\alpha} - {\tau_\alpha} \in T_{<\alpha} \eqdef \bigoplus_{\beta < \alpha} T_\beta\;.
\end{equ}
Furthermore, $\Gamma \one = \one$ for every $\Gamma \in G$.
\end{itemize}
\end{definition}

The prime example of a regularity structure one should keep in mind is the one associated to Taylor polynomials
on space-time $\R^{d+1}$.
In this case, the space $T$ is given by all polynomials in $d+1$ indeterminates $X_0,\ldots,X_d$, with $X_0$ representing
the ``time'' coordinate. It comes with a canonical basis given by all monomials of the type $X^k = X_0^{k_0}\cdots X_d^{k_d}$ with
$k$ an arbitrary multiindex. The basis vector $\one$ is the one corresponding to the zero multiindex.
The space $T$ has a natural grading by postulating that the homogeneity of $X^k$ is $|k| = 2k_0 + \sum_{i\neq 0} k_i$.
(But of course any other ``reasonable'' way of counting degrees would also do, including the usual way.)
In the case of the polynomial regularity structure, the structure group $G$ is simply given by $\R^{d+1}$, endowed
with addition, and acting on monomials by
\begin{equ}[e:defGamma]
\hat\Gamma_h X^k = (X-h)^k = (X_0 - h_0)^{k_0}\cdots (X_d - h_d)^{k_d}\;.
\end{equ}
One has of course the identity $\hat\Gamma_{h+h'} = \hat\Gamma_h \circ \hat\Gamma_{h'}$, and it
is immediate that all of the axioms of a regularity structure are satisfied in this case.

In the case of polynomials, there is a natural ``realisation'' of the structure $\TT$ at each space-time point $z$,
which is obtained by turning an abstract polynomial into the corresponding concrete polynomial (viewed now as a real-valued
function on $\R^{d+1}$) based at $z$. In other words, we naturally have a family of linear maps $\Pi_z\colon T \to \CC^\infty(\R^d)$ 
given by
\begin{equ}[e:defPi]
\bigl(\Pi_z X^k\bigr)(z') = (z'_0 - z_0)^{k_0}\cdots (z'_d - z_d)^{k_d}\;.
\end{equ}
It is immediate that the group $G$ transforms these maps into each other in the sense that
\begin{equ}
\Pi_z \hat \Gamma_{h} = \Pi_{z+h}\;.
\end{equ}
It is furthermore an immediate consequence of the scaling properties of monomials that the maps $\Pi_z$ and the
representation $h \mapsto \hat\Gamma_h$ of $\R^{d+1}$ are ``compatible'' with our grading for the model space $T$.
More precisely, one has
\begin{equ}
\scal{\varphi_z^\lambda, \Pi_z X^k} = \lambda^{|k|} \scal{\varphi, \Pi_0 X^k}\;,\qquad
\|\hat \Gamma_h X^k\|_\ell = C_{k,\ell} |h|^{|k|- \ell}\;,
\end{equ}
for some constants $C_{k,\ell}$ and every $\ell \le |k|$. Here, $\scal{\cdot,\cdot}$ denotes again the usual
$L^2$-scalar product.

These observations suggest the following definition of a ``model'' for $\TT$, where we impose properties similar
to the ones we just found for the polynomial model. A model always requires the specification of an ambient space,
together with a possibly inhomogeneous scaling. For definiteness, we will fix our ambient space to be $\R^{d+1}$ endowed
withe the parabolic scaling as above, but other scalings are treated in virtually the same way. We also denote by 
$\CS'$ the space of all distributions (even though we really consider distributions over compactly supported test functions,
but the letter $\CD$ is reserved for a different usage below). We also denote by $L(E,F)$ the set of all continuous
linear maps between the topological vector spaces $E$ and $F$.

\begin{definition}\label{def:model}
Given a regularity structure $\TT$, a \textit{model} for $\TT$ consists of maps 
\begin{equ}
\R^{d+1} \ni z \mapsto \Pi_z \in L(T, \CS')\;,\qquad \R^{d+1}\times \R^{d+1} \ni (z, z') \mapsto \Gamma_{zz'} \in G\;,
\end{equ}
satisfying the algebraic compatibility conditions
\begin{equ}[e:alg]
\Pi_z \Gamma_{zz'} = \Pi_{z'}\;,\qquad \Gamma_{zz'} \circ \Gamma_{z'z''} = \Gamma_{zz''}\;,
\end{equ}
as well as the analytical bounds
\begin{equ}[e:ana]
|\scal{\Pi_z \tau, \varphi_z^\lambda}| \lesssim \lambda^{\alpha}\|\tau\|\;,\qquad
\|\Gamma_{zz'} \tau\|_\beta \lesssim |z-z'|^{\alpha-\beta} \|\tau\|\;.
\end{equ}
Here, the bounds are imposed uniformly over all $\tau \in T_\alpha$, all $\beta < \alpha \in A$, and all test functions
$\varphi \in \CB_{r}$ with $r = \inf A$. They are imposed locally uniformly in $z$ and $z'$.
\end{definition}

\begin{remark}
These definitions suggest a natural topology for the space $\MM$ of all models for a given regularity structure, generated
by the following family of pseudo-metrics indexed by compact sets $K$:
\begin{equ}[e:metricModel]
\sup_{z \in K} \Bigl(\sup_{\varphi,\lambda,\alpha,\tau} \lambda^{-\alpha} |\scal{\Pi_z \tau - \bar \Pi_z \tau, \varphi_z^\lambda}|
+ \sup_{|z-z'| \le 1} \sup_{\alpha,\beta,\tau}|z-z'|^{\beta-\alpha} \|\Gamma_{zz'} \tau - \bar\Gamma_{zz'} \tau\|_\beta \Bigr)\;.
\end{equ}
Here the inner suprema run over the same sets as before, but with $\|\tau\| = 1$.
\end{remark}

\subsection{H\"older classes}

It is clear from the above discussion that if $\TT$ is the polynomial structure, $\Pi$ is defined
as in \eqref{e:defPi}, and $\Gamma_{zz'} = \hat \Gamma_{z'-z}$ with $\hat \Gamma_h$ as in \eqref{e:defGamma}, then 
$(\Pi,\Gamma)$ is a model for $\TT$ in the sense of Definition~\ref{def:model}.
Given an arbitrary regularity structure $\TT$ and an arbitrary model $(\Pi,\Gamma)$, it is now natural to define
the corresponding ``H\"older spaces'' as spaces of distributions that can locally (near any space-time point $z$) 
be approximated by $\Pi_z \tau$ for some $\tau \in T$. This would be the analogue to the statement that a smooth
function is one that can locally be approximated by a polynomial. 

There is however one major difference with the case of smooth functions. It is of course the case that if 
$f$ is smooth, then the coefficients of the Taylor expansion of $f$ at any point are uniquely determined by the
behaviour of $f$ in the vicinity of that point. This is in general \textit{not} the case anymore in the context of
the framework we just described. To appreciate this fact, consider the following example. Fix $\alpha \in (0,1)$ and
$m \in \N$, and take for $\TT$ the regularity structure where $A = \{0,\alpha\}$, $T_0 \cong \R$ with basis vector $\one$,
$T_\alpha \cong \R^m$ with basis vectors
$(e_i)_{i\le m}$, and structure group $G \cong \R^m$ acting on $T$ via
\begin{equ}
\hat \Gamma_h e_i = e_i - h_i \one\;.
\end{equ}
Let then $W$ be an $\R^m$-valued $\alpha$-H\"older continuous function defined on the ambient space
and set
\begin{equ}
\Pi_z \one = 1\;,\qquad \bigl(\Pi_z e_i\bigr)(z') = W_i(z') - W_i(z)\;,\qquad 
\Gamma_{zz'} = \hat \Gamma_{W(z) - W(z')}\;. 
\end{equ}
Again, it is straightforward to verify that this does indeed define a model for $\TT$.
In fact, setting $m=1$ and $W = \eta$, this is precisely the structure one would use to
formalise the expansion \eqref{e:approxu}.

Let now $F \colon \R^m \to \R$ be a smooth function and consider the function $f$ on the ambient space
given by $f(z) = F(W(z))$. For any $z$, we furthermore set
\begin{equ}
T \ni \hat f(z) = F(W(z))\,\one + \sum_{i=1}^m (\partial_i F)(W(z))\,e_i \;.
\end{equ}
It then follows immediately from the usual Taylor expansion of $F$ and the definition of the model $(\Pi,\Gamma)$ 
that one has the bound
\begin{equ}[e:boundHolder]
\bigl|f(z') - \bigl(\Pi_z \hat f(z)\bigr)(z')\bigr| \lesssim |z-z'|^{2\alpha}\;,
\end{equ}
so that in this context and with respect to this specific model, the function $f$ behaves as if it were
of class $\CC^{2\alpha}$ with ``Taylor series'' given by $\hat f$.
In the case where the underlying space is one-dimensional, this is precisely the insight exploited in the theory
of rough paths \cite{MR2314753,MR2036784,MR2604669} in order to develop a pathwise approach to stochastic calculus. 
More specifically, the perspective given here (i.e.\ controlling functions via analogues to Taylor expansion) 
is that of the theory of controlled  rough paths developed in \cite{Max}.

It is now very natural to ask whether, just like in the case of smooth functions, a bound of the type 
\eqref{e:boundHolder} is sufficient to uniquely specify $\hat f(z)$ for every point $z$. Unfortunately,
the answer to this question is that ``it depends''. The reason is that while \eqref{e:ana} imposes an
upper bound on the behaviour of $\Pi_z$ in the vicinity of $z$, it does \textit{not} impose
any corresponding lower bound. For example, the function $W(z) = 0$ is a perfectly admissible $\alpha$-H\"older
continuous function that we could have used to build our model. In that case, the value of the $e_i$-component in $\hat f$
is completely irrelevant for the bound \eqref{e:boundHolder}, so that uniqueness of the ``Taylor series'' fails.
Suppose on the other hand that the underlying space is one-dimensional, that $\alpha \in ({1\over 4}, {1\over 2})$, and
that $W$ is a typical sample path of a Brownian trajectory. In this case (and also 
in the more general case of fractional Brownian motion), it was shown in \cite[Thm~3.4]{Natesh} (see also \cite{Tom} for
a slightly more general statement)
that a bound of the type \eqref{e:boundHolder} is indeed sufficient to 
uniquely determine all the coefficients of $\hat f$ (at least for almost all Brownian trajectories).

\begin{remark}
The fact that $\hat f$ is uniquely determined by $f$ in the Brownian case can be interpreted
as an analogue to the fact that the Doob-Meyer decomposition of a semimartingale is unique.
Since the statement given in \cite{Natesh} is quantitative, it can be interpreted as a 
deterministic analogue to Norris's lemma, of which various incarnations can be found in
\cite{Bismut1,Bismut2,KSAMI,KSAMII,KSAMIII,Nor86SMC}. Another deterministic analogue to this
lemma was previously obtained in a more restrictive context in \cite[Thm~7.1]{ErgodicBig}.
\end{remark}

Consider now a sequence $W^\eps$ of smooth (random) functions so that $W^\eps$ converges to Brownian motion
in $\CC^\alpha$ as $\eps \to 0$. For definiteness, take for $W_\eps$ piecewise linear interpolations on
a grid of size $\eps$. Then, if we know \textit{a priori} that we have a bound of the 
type \eqref{e:boundHolder} with a proportionality constant of order $1$, this determines the coefficients
of $\hat f$ ``almost uniquely'' up to an error of order about $\eps^{2\alpha - {1\over 2}}$.

What this discussion suggests is that we should really reverse our point of view from what 
we are used to: instead of fixing a function and asking whether it has a certain H\"older regularity
by checking whether it is possible to find a ``Taylor expansion'' at each point satisfying a bound
of the type \eqref{e:boundHolder}, we should take the candidate expansion as our fundamental object and
ask under which condition it does indeed approximate one single function / distribution around
each point at the prescribed order. More precisely, in order to define suitable classes of ``regular distributions''
we should address the following question. Fix some $\gamma > 0$ (the order of our ``Taylor expansion'') and consider
a function $f \colon \R^{d+1} \to T_{<\gamma}$. Under which assumptions can we find a distribution $\zeta$
such that $\zeta$ ``looks like'' the distribution $\Pi_z f(z)$ (in a suitable sense) near every point $z$?
We claim that the ``right'' answer is given by the following definition.

\begin{definition}\label{def:Dgamma}
Given a regularity structure $\TT$ and a model $(\Pi,\Gamma)$ as above, we define $\CD^\gamma$ as
the space of functions $f \colon \R^{d+1} \to T_{<\gamma}$ such that the bound
\begin{equ}[e:defDgamma]
\|f(z) - \Gamma_{zz'}f(z')\|_\alpha \lesssim |z-z'|^{\gamma - \alpha}\;.
\end{equ}
holds for every $\alpha < \gamma$, locally uniformly in $z$ and $z'$.
\end{definition}

\begin{remark}
This definition
makes sense and is non-empty even for negative $\gamma$, as long as $\gamma > \inf A$.
\end{remark}

\begin{remark}
The notation $\CD^\gamma$ is really an abuse of notation, since even for a given regularity structure
there isn't one single space  $\CD^\gamma$, but a whole collection of them, one for each model $(\Pi,\Gamma) \in \MM$.
More formally, one should really consider the space $\MM \ltimes \CD^{\gamma}$ consisting of
pairs $((\Pi,\Gamma),f)$ such that $f$ belongs to the space $\CD^\gamma$ based on the model
$(\Pi,\Gamma)$. The space $\MM \ltimes \CD^{\gamma}$ also comes with a natural topology generated again by 
a family of pseudo-metrics indexed by compact sets $K$. These metrics are given by \eqref{e:metricModel} for the distances
between models $(\Pi,\Gamma)$ and $(\bar \Pi,\bar \Gamma)$, and by
\begin{equ}
\sup_{z, z' \in K \atop |z-z'| \le 1} \sup_{\alpha < \gamma} |z-z'|^{\alpha-\gamma} \|f(z) - \Gamma_{zz'} f(z') - {\bar f}(z) + \bar \Gamma_{zz'} {\bar f}(z')\|_{\alpha} \;,
\end{equ}
for the distance between the corresponding elements in $\CD^\gamma$.
\end{remark}

In the case where $\TT$ is the polynomial regularity structure and $(\Pi,\Gamma)$
are the usual Taylor polynomials as above, one can see that this definition coincides with the usual definition of $\CC^\gamma$
(except at integer values where $\CD^1$ describes Lipschitz continuous functions, etc).
In this case, the component $f_0(z) = \scal{\one, f(z)}$ of $f(z)$ in $T_0$ is the only reasonable 
candidate for the function represented by $f$. Furthermore, $\scal{\one,\Gamma_{zz'}f(z')}$ is 
nothing but the candidate Taylor expansion of $f$ around $z'$, evaluated at $z$. The bound \eqref{e:defDgamma}
with $\alpha = 0$ is then just a statement of the fact that $f_0$ is of class $\CC^\gamma$ and that 
$f(z)$ is its Taylor series of order $\gamma$ at $z$. The corresponding bounds for $\alpha > 0$ then follow 
immediately, since they merely state that the $\alpha$th derivative of $f_0$ is of class $\CC^{\gamma - \alpha}$.

\subsection{The reconstruction operator}

The situation is much less straightforward when the model space $T$ contains components of negative
homogeneity. In this case, the bounds \eqref{e:ana} allow the model $\Pi_z$ to consist of genuine
distributions and we do not anymore have an obvious candidate for the distribution represented
by $f$. The following result shows that such a distribution nevertheless always exists and is unique
as soon as $\gamma > 0$. This also provides an \textit{a posteriori} justification for our definition
of the spaces $\CD^\gamma$.

\begin{theorem}\label{theo:reconstruction}
Consider a regularity structure $\TT = (A,T,G)$ and fix $\gamma > r = \inf A$. Then, there exists a continuous map
$\CR \colon \MM \ltimes \CD^\gamma \to \CS'$ (the ``reconstruction map'') with the property that
\begin{equ}[e:bound]
\bigl|\bigl(\CR (\Pi,\Gamma,f) - \Pi_z f(z)\bigr)(\varphi_z^\lambda)\bigr| \lesssim \lambda^\gamma\;,
\end{equ}
uniformly over $\lambda \in (0,1]$ and $\varphi \in \CB_r$, and locally uniformly over $z \in \R^{d+1}$.
Furthermore, for any given model $(\Pi,\Gamma)$, the map $f \mapsto \CR (\Pi,\Gamma,f)$ is linear.
If $\gamma > 0$, the map $\CR$ is uniquely specified by the requirement \eqref{e:bound}.
\end{theorem}

\begin{remark}
In the sequel, we will always consider $(\Pi,\Gamma)$ as fixed and view $\CR$ as a linear map, writing
$\CR f$ instead of $\CR(\Pi,\Gamma,f)$. The above notation does however make it plain that the full map $\CR$ is
\textit{not} a linear map. It is important to note that the continuity mentioned in the statement of 
Theorem~\ref{theo:reconstruction} is with respect to the topology of all of $\MM \ltimes \CD^\gamma$ and not
just that of $\CD^\gamma$ for a fixed model. This is very important when studying the stability of the solution
map to our stochastic PDEs under perturbations / approximations of the driving noise $\xi$.
\end{remark}

\begin{remark}
An important special case is given by situations where $\Pi_z \tau$ is actually a continuous function for every
$\tau \in T$ and every $z$. If $T$ contains components of negative homogeneity, this is certainly not the case for
every model in $\MM$, but there are certainly \textit{some} models for which this is true. In fact, all models
that are of interest of us will turn out to be limits of such models. Then, it turns out that $\CR f$ is also a
continuous function and one simply has 
\begin{equ}[e:formulaRf]
\bigl(\CR f\bigr)(z) = \bigl(\Pi_z f(z)\bigr)(z)\;.
\end{equ}
In the general case, this formula makes of course no sense since $\Pi_z f(z)$ is a distribution and cannot be
evaluated at $z$.
\end{remark}

\begin{remark}
Another special case is given by the situation where $f(z) = \Gamma_{z0}\tau$ for some fixed $\tau \in T$. In
this case, one has $f(z) - \Gamma_{zz'} f(z') = 0$, so that $f \in \CD^\gamma$ for every $\gamma$. These
are the functions that play the role of polynomials in our theory. In this case, one has $\CR f = \Pi_0 \tau$.
\end{remark}

\begin{remark}
We made a slight abuse of notation here since there is really a family of operators $\CR^\gamma$, one for
each regularity. However, this abuse is justified by the following consistency relation. Given $f \in \CD^\gamma$
and $\tilde \gamma < \gamma$, one can always construct $\tilde f$ by projecting $f(z)$ onto $T_{<\tilde \gamma}$
for every $z$. It turns out that one then necessarily has $\tilde f \in \CD^{\tilde \gamma}$ and $\CR \tilde f = \CR f$,
provided that $\tilde \gamma > 0$. This is also consistent with \eqref{e:formulaRf} since, if $\Pi_z \tau$ is a continuous
function and the homogeneity of $\tau$ is strictly positive, then $ \bigl(\Pi_z \tau\bigr)(z) = 0$.
\end{remark}

We refer to \cite[Thm~3.10]{Regularity} for a full proof of Theorem~\ref{theo:reconstruction} and to \cite{Notes} for
a simplified proof that only gives continuity in each ``fiber'' $\CD^\gamma$. The main idea
is to use a basis of compactly supported wavelets to construct approximations $\CR^n$ in such a 
way that our definitions can be exploited in a natural way to compare $\CR^{n+1}$ with $\CR^n$ and
show that the sequence of approximations is Cauchy in a suitable space of distributions $\CC^\alpha$.
In the most important case when $\gamma >0$, it turns out that while the existence of a map $\CR$ with
the required properties is highly non-trivial, its uniqueness is very easy to see. Given an element
$(\Pi,\Gamma,f) \in \MM \ltimes \CD^\gamma$, assume that we have two candidates $\zeta$ and $\bar \zeta$
for $\CR (\Pi,\Gamma,f)$. It then follows immediately from \eqref{e:bound} that one has the bound
\begin{equ}[e:diffzeta]
\bigl|\bigl(\zeta - \bar \zeta\bigr)(\varphi_z^\lambda)\bigr| \lesssim \lambda^\gamma\;,
\end{equ}
uniformly over $\varphi$ and $\lambda$, and locally uniformly over $z$.
Let now $\psi$ be a smooth test function with compact support and write $\psi^\lambda = \varphi_0^\lambda \star \psi$
for a test function $\varphi$ as above which is furthermore such that $\int \varphi(z)\,dz = 1$.
It then follows from the definition of a distribution that $\lim_{\lambda \to 0} \zeta(\psi^\lambda) = \zeta(\psi)$
and similarly for $\bar \zeta$. On the other hand, one has the identity
\begin{equ}
\zeta(\psi^\lambda) = \zeta \Bigl(\int \psi(z) \varphi_z^\lambda(\cdot) \,dz\Bigr) = \int \psi(z) \zeta(\varphi_z^\lambda)\,dz\;,
\end{equ}
so that, as a consequence of \eqref{e:diffzeta}, one has 
\begin{equ}
|\zeta(\psi^\lambda) - \bar\zeta(\psi^\lambda)| = \Bigl|\int \psi(z) \bigl(\zeta - \bar \zeta\bigr)(\varphi_z^\lambda)\,dz\Bigr| 
\lesssim \lambda^\gamma\;.
\end{equ}
Sending $\lambda$ to $0$ and using the fact that $\gamma > 0$, it immediately 
follows that $\zeta(\psi) = \bar\zeta(\psi)$. Since $\psi$ was arbitrary, we conclude that indeed $\zeta = \bar \zeta$. 
If $\gamma \le 0$ on the other hand, it is clear that $\CR$ cannot be uniquely determined by \eqref{e:bound}, since
this bound remains unchanged if we add to $\CR$ any distribution in $\CC^\gamma$. The existence of $\CR$ in the case $\gamma < 0$
is however still a non-trivial result since in general one has $\CR f \not \in \CC^\gamma$!

\section{Regularity structures for SPDEs}
\label{sec:SPDE}

We now return to the problem of providing a robust well-posedness theory for stochastic
PDEs of the type \eqref{e:Burgers}, \eqref{e:SPDEs}, \eqref{e:KPZ}, or even just \eqref{e:SPDE}. 
Our aim is to build a suitable
regularity structure for which we can reformulate our SPDE as a fixed point problem in $\CD^\gamma$ for
a suitable value of $\gamma$.

\begin{remark}
Actually, it turns out that since we are interested in Cauchy problems, there will always be some
singularity at $t=0$. This introduces additional technical complications and requires us to work
in spaces $\CD^{\gamma,\eta}$ that are really weighted versions of the spaces $\CD^\gamma$ with weights
allowing for some singularity (essentially of strength $\eta$) on the hyperplane $t=0$. While
this is technically quite tricky, it does not introduce any conceptual difficulty, so we chose not
to dwell on this problem in the present review.
\end{remark}

\subsection{General construction of the model space}
\label{sec:genConst}

Our first task is to construct the model space $T$. Since we certainly want to be able to represent arbitrary
smooth functions (for example in order to be able to take into account the contribution of the initial condition),
we want $T$ to contain the space $\bar T$ of abstract polynomials in $d+1$ indeterminates endowed with
the parabolic grading described in Section~\ref{sec:def}. Since the noise $\xi$ cannot be adequately represented by
polynomials, we furthermore add a basis vector $\Xi$ to $T$, which we postulate to have some homogeneity $\alpha < 0$
such that $\xi \in \CC^\alpha$. In the case of space-time white noise, we would choose $\alpha = -{d\over 2} - 1 - \kappa$
for some (typically very small) exponent $\kappa > 0$. In the case of purely spatial white noise, we would choose 
$\alpha = -{d\over 2} - \kappa$, etc.

At this stage, the discussion following \eqref{e:approxu} suggests that if our structure $T$ contains a
basis vector $\tau$ of homogeneity $\beta$ representing some distribution 
$\eta$ involved in the description of the right hand side of our equation, then it should also 
contain a basis vector of homogeneity $\beta + 2$ (the ``$2$'' here comes from the fact that convolution with the heat
kernel yields a gain of $2$ in regularity) representing the distribution $K \star \eta$ involved
in the description of the solution to the equation. Let us denote this new basis vector by $\CI(\tau)$,
where $\CI$ stands for ``integration''. In the special case where $\tau \in \bar T$, so that it represents 
an actual polynomial, we do not need any new symbol since $K$ convolved with a polynomial
yields a smooth function. One way of formalising this is to simply postulate that $\CI(X^k) = 0$ for every
multiindex $k$. 

\begin{remark}
For consistency, we will also always assume that $\int K(z)Q(z)\,dz = 0$ for all polynomials $Q$ 
of some fixed, but sufficiently high, degree. Since $K$ is an essentially arbitrary truncation of 
the heat kernel, we can do this without loss of generality.
\end{remark}

If the right hand side of our equation involves the spatial derivatives of the solution,
then, for each basis vector $\tau$ of homogeneity $\beta$ representing some distribution 
$\eta$ appearing in the description of the solution,
we should also have a basis vector $\DD_i \tau$ of homogeneity $\beta - 1$ representing $\partial_i \eta$
and appearing in the description of the derivative of the solution in the direction $x_i$.

Finally, if the right hand side of our equation involves a product between two terms
$F$ and $\bar F$, and if basis vectors $\tau$ and $\bar \tau$ respectively are involved in their description,
then we should also have a basis vector $\tau \bar \tau$ which would be involved in the description of the 
product. If $\tau$ and $\bar \tau$ represent the distributions $\eta$ and $\bar \eta$ respectively,
then this new basis vector represents the distribution $\eta \bar \eta$, whatever this actually means. 
Regarding its homogeneity, by analogy with the case of polynomials, it is natural to impose that the
homogeneity of $\tau \bar \tau$ is the sum of the homogeneities of its two factors.

This suggests that we should build $T$ by taking as its basis vectors some formal expressions built from the symbols 
$X$ and $\Xi$, together with the operations $\CI(\cdot)$, $\DD_i$, and multiplication. Furthermore, the natural way
of computing the homogeneity of a formal expression in view of the above is to associate homogeneity 
$2$ to $X_0$, $1$ to $X_i$ for $i \neq 0$,
$\alpha$ to $\Xi$, $2$ to $\CI(\cdot)$, and $-1$ to $\DD_i$, and to simply add the homogeneities of all symbols appearing 
in any given expression. Denote by $\CF$ the collection of all formal expressions that can be 
constructed in this way and denote by $|\tau|$ the homogeneity of $\tau \in \CF$,
so we have for example
\begin{equ}
\bigl|X_i \Xi\bigr| = \alpha + 1\;,\qquad  \bigl|\CI(\Xi)^2\CI(X_i \DD_j\CI(\Xi))\bigr| = 3\alpha + 8\;,\qquad \text{etc.}
\end{equ}
We note however that if we simply took for $T$ the space of linear combinations of \textit{all} elements in $\CF$ then, 
since $\alpha < 0$, there would be basis vectors of arbitrarily negative homogeneity,
which would go against Definition~\ref{def:regStruct}. What saves us is that most formal expressions
are not needed in order to formulate our equations as fixed point problems. For example, the expression $\Xi^2$
is useless since we would never try to square the driving noise. Similarly, if we consider \eqref{e:AC}, then 
$\CI(\Xi)$ is needed for the description of the solution, which implies that $\CI(\Xi)^2$ and $\CI(\Xi)^3$
are needed to describe the right hand side, but we do not need $\CI(\Xi)^4$ for example.

\subsection{Specific model spaces}\label{sec:spaces}

This suggests that we should take $T$ as the linear combinations of only those formal expressions $\tau \in \CF$ that 
are actually expected to appear in the description of the solution to our equation or its right hand side.
Instead of trying to formulate a general construction (see \cite[Sec.~8.1]{Regularity} for such an attempt), 
let us illustrate this by a few examples.
We first focus on the case of \eqref{e:AC} and we construct subsets $\CU$ and $\CV$ of $\CF$ that are used in the 
description of the solution and the right hand side of the equation respectively. These are defined as the smallest
subsets of $\CF$ with the following properties:
\begin{equ}[e:UVAC]
\CT \subset \CU \cap \CV\;,\quad
\{\CI(\tau)\,:\, \tau \in \CV \setminus \CT\} \subset \CU\;,\quad
\{\Xi\} \cup \{\tau_1 \tau_2 \tau_3\,:\, \tau_i \in \CU\} \subset \CV\;,
\end{equ}
where we used the notation $\CT = \{X^k\}$ with $k$ running over all multiindices, so that the space of Taylor polynomials
$\bar T$ is the linear span of $\CT$.
We then define $T$ as the space of all linear combinations of elements of $\CV$. Note that $\CU \subset \CV$ by the first and 
last properties, so that linear combinations of elements of $\CU$ also belong to $T$ and we denote by $T_\CU$ the corresponding
subspace of $T$.
This construction is such that if we have any function $\Phi \colon \R^{d+1} \to T_\CU$, then we can define in a natural
way a function $\Xi - \Phi^3 \colon \R^{d+1} \to T$ by the last property. 
Furthermore, by the second property, one has again $\CI(\Xi - \Phi^3) \colon \R^{d+1} \to T_\CU$,
which suggests that $T$ is indeed sufficiently rich to formulate a fixed point problem mimicking the mild formulation
of \eqref{e:AC}. Furthermore, one has

\begin{lemma}\label{lem:critAC}
If $\CU$ and $\CV$ are the smallest subsets of $\CF$ satisfying \eqref{e:UVAC} and 
one has $|\Xi| > -3$ then, for every $\gamma > 0$, the
set $\{ \tau \in \CU\,:\, |\tau| < \gamma\}$ is finite. 
\end{lemma}

Note that in terms of the regularity of space-time white noise, the condition $\alpha > -3$ corresponds to 
the restriction $d < 4$. It is well-known \cite{MR678000,MR643591} that $4$ is the critical dimension for the static analogue
to this model, which strongly suggests that this is indeed the ``right'' condition.

In the case of the KPZ equation, we similarly set 
\begin{equ}[e:UVKPZ]
\CT \subset \CU \cap \CV\;,\quad
\{\CI(\tau)\,:\, \tau \in \CV \setminus \CT\} \subset \CU\;,\quad
\{\Xi\}\cup \{\DD \tau_1\cdot  \DD \tau_2\,:\, \tau_i \in \CU\} \subset \CV\;.
\end{equ}
This time, $\CU \not \subset \CV$, so we define $T$ as the space of all linear combinations of elements
in $\CU \cup \CV$. Again, we have

\begin{lemma}\label{lem:critKPZ}
If $\CU$ and $\CV$ are the smallest subsets of $\CF$ satisfying \eqref{e:UVKPZ} and 
one has $|\Xi| > -2$ then, for every $\gamma > 0$, the
set $\{ \tau \in \CU \cup \CV\,:\, |\tau| < \gamma\}$ is finite. 
\end{lemma}

This time, the condition $\alpha > -2$ corresponds to the restriction $d < 2$, which again makes sense since
$2$ is the critical dimension for the KPZ equation \cite{KPZOrig}.
The last example we would like to consider is the class of SPDEs \eqref{e:SPDE}. In this case, the right hand side
is not polynomial. However, we can apply the same methodology as above as if the nonlinear functions $f$ and $g$
were simply polynomials of arbitrary degree. We thus impose $\CT \subset \CU \cap \CV$ and $\{\CI(\tau)\,:\, \tau \in \CV \setminus \CT\}$
as before, and then further impose that
\begin{equ}
\Big\{\Xi\prod_{i=1}^m \tau_i\,:\, m \ge 1\;\&\; \tau_i \in \CU\Big\}\cup 
\Big\{\prod_{i=1}^m \tau_i\,:\, m \ge 1\;\&\; \tau_i \in \CU\Big\} \subset \CV\;.
\end{equ}
Again, we have $\CU \subset \CV$ and we define $T$ as before. Furthermore, it is straightforward to verify that
the analogue to lemmas~\ref{lem:critAC} and \ref{lem:critKPZ} holds, provided that $|\Xi| > -2$.

\subsection{Construction of the structure group}
\label{sec:group}

Now that we have some idea on how to construct $T$ for the problems that are of interest to us (with a slightly
different construction for each class of models but a clear common thread), we 
would like to build a corresponding structure group $G$. In order to give a motivation for the definition of $G$,
it is very instructive to simultaneously think about the structure of the corresponding models. 
Let us first consider some smooth driving noise, which we call $\xi_\eps$ to distinguish it from the 
limiting noise $\xi$. At this stage however, this should be thought of as simply a fixed smooth (or at 
least continuous) function. In view of the discussion of Section~\ref{sec:genConst}, for each of the
model spaces built in Section~\ref{sec:spaces}, we can associate to any smooth function $\xi_\eps$ a 
linear map $\PPi\colon T \to \CC^\infty(\R^{d+1})$ in the following way. We set
\minilab{e:defPPi}
\begin{equ}[e:defPPi1]
\bigl(\PPi X_i\bigr)(z) = z_i\;,\qquad \bigl(\PPi \Xi\bigr)(z) = \xi_\eps(z)\;,
\end{equ}
and we then define $\PPi$ recursively by
\minilab{e:defPPi}
\begin{equ}[e:defPPi2]
\PPi \CI(\tau) = K \star \PPi \tau\;,\qquad
\PPi \DD_i\tau = \partial_i \PPi \tau\;,\qquad 
\PPi(\tau \bar \tau) = \bigl(\PPi \tau\bigr)\cdot \bigl(\PPi\bar \tau\bigr)\;,
\end{equ}
where $\cdot$ simply denotes the pointwise product between smooth functions. At this stage, it is however
not clear how one would build an actual model in the sense of Definition~\ref{def:model} associated to $\xi_\eps$.
It is natural that one would set
\minilab{e:defPi}
\begin{equ}[e:defPi1]
\bigl(\Pi_z X_i\bigr)(z') = z_i' - z_i\;,\qquad \bigl(\Pi_z \Xi\bigr)(z') = \xi_\eps(z')\;,
\end{equ}
and then 
\minilab{e:defPi}
\begin{equ}[e:defPi2]
\Pi_z \DD_i\tau = \partial_i \Pi_z \tau\;,\qquad 
\Pi_z(\tau \bar \tau) = \bigl(\Pi_z \tau\bigr)\cdot \bigl(\Pi_z \bar \tau\bigr)\;.
\end{equ}
It is less clear \textit{a priori} how to define $\Pi_z \CI(\tau)$. The problem is that if
we simply set $\Pi_z \CI(\tau) = K \star \Pi_z \tau$, then the bound \eqref{e:ana} would typically
no longer be compatible with the requirement that $|\CI(\tau)| = |\tau| + 2$. One way to circumvent
this problem is to simply subtract the Taylor expansion of $K \star \Pi_z \tau$ around $z$ up to the
required order. We therefore set
\minilab{e:defPi}
\begin{equ}[e:defPi3]
\bigl(\Pi_z \CI(\tau)\bigr)(z') = \bigl(K \star \Pi_z \tau\bigr)(z') - \sum_{|k| < |\tau| + 2} {(z'-z)^k \over k!}
\bigl(D^{(k)}K \star \Pi_z \tau\bigr)(z)\;.
\end{equ}
It can easily be verified (simply proceed recursively) 
that if we define $\Pi_z$ in this way and $\PPi$ as in \eqref{e:defPPi} then, for every $z$, one can find a linear map 
$F_z \colon T \to T$ such that $\Pi_z = \PPi F_z$. In particular, one has $\Pi_{z'} = \Pi_z F_z^{-1}F_{z'}$. 
Furthermore, $F_z$ is ``upper triangular'' with the identity 
on the diagonal in the sense of \eqref{e:coundGroup}.
It is also easily seen by induction that the matrix elements of $F_z$ are all given by some polynomials in 
$z$ and in the quantities $\bigl(D^{(k)}K \star \Pi_z \tau\bigr)(z)$.

This suggests that we should take for $G$ the set of all linear maps that can appear in this fashion. It is
however not clear in principle how to describe $G$ more explicitly and it is also not clear that it even forms
a group. In order to describe $G$, it is natural to introduce some abstract space $T_+$ (different from $T$)
which is given by all possible polynomials in $d+1$ commuting variables $\{Z_i\}_{i=0}^d$ as well as 
countably many additional commuting variables $\{\CJ_k(\tau)\,:\, \tau \in (\CU \cup\CV) \setminus \CT\; \&\; |k| < |\tau|+2\}$.
One should think of $Z_i$ as representing $z_i$ and $\CJ_k(\tau)$ as representing $\bigl(D^{(k)}K \star \Pi_z \tau\bigr)(z)$,
so that the matrix elements of $F_z$ are represented by elements of $T_+$.
There are no relations between these coefficients, which suggests that elements of $G$ are described
by an arbitrary morphism $f \colon T_+ \to \R$, i.e.\ an arbitrary linear map which furthermore satisfies
$f(\sigma \bar \sigma) = f(\sigma)\, f(\bar \sigma)$, so that it is uniquely determined by $f(Z_i)$ and $f(\CJ_k(\tau))$.

Given any linear map $\Delta \colon T \to T \otimes T_+$ and a morphism $f$ as above, one can then define
a linear map $\hat \Gamma_f \colon T \to T$ by
\begin{equ}
\hat \Gamma_f \tau = \bigl(I \otimes f\bigr) \Delta \tau\;.
\end{equ}
(Here we identify $T$ with $T \otimes \R$ in the obvious way.) The discussion given above then suggests that 
it is possible to construct $\Delta$ in such a way that if we define $f_z$ by 
\begin{equ}[e:deffz]
f_z(Z_i) = z_i\;,\qquad f_z(\CJ_k(\tau)) = \bigl(D^{(k)}K \star \Pi_z \tau\bigr)(z)\;,
\end{equ}
then one has $\hat \Gamma_{f_z} = F_z$.
The precise definition of $\Delta$ is irrelevant for our discussion, but a recursive description of it can
easily be recovered simply by comparing \eqref{e:defPi} to \eqref{e:defPPi}. In particular, it is possible to
show that
$\Delta \tau$ is of the form
\begin{equ}[e:propDelta]
\Delta \tau = \tau \otimes \one + \sum_{i} c_i^\tau \tau_i \otimes \sigma_i\;,
\end{equ}
for some expressions $\tau_i \in T$ with $|\tau_i| < |\tau|$ and for some non-empty monomials $\sigma_i \in T_+$
such that $|\sigma_i| + |\tau|_i = |\tau|$. Here, we associate a homogeneity to elements in $T_+$ by
setting $|Z_0| = 2$, $|Z_i| = 1$ for $i \neq 0$, and $|\CJ_k(\tau)| = |\tau|+2-|k|$.

In particular, we see that if we let $e \colon T_+ \to \R$ be the trivial morphism for which $e(Z_i) = e(\CJ_k(\tau)) = 0$,
so that one only has $e(\one) = 1$ where $\one$ is the empty product, then 
\begin{equ}[e:defe]
\hat \Gamma_e \tau = \tau\;.
\end{equ}
The important fact for our purpose is
the following, a proof of which can be found in \cite[Sec.~8]{Regularity}. Here, we denote by $\CM \colon T_+ \otimes T_+ \to T_+$ the
multiplication operator $\CM (\sigma \otimes \bar \sigma) = \sigma \bar \sigma$ and by $I$ the identity. 

\begin{lemma}
There exists a map $\Delta^+\colon T_+ \to T_+ \otimes T_+$ such that the following identities hold:
\begin{equs}[2][e:idenDelta]
\Delta^+ (\sigma \bar \sigma) &= \bigl(\Delta^+ \sigma\bigr) \cdot \bigl(\Delta^+\bar \sigma\bigr)\;,\quad&
(\Delta \otimes I)\Delta &= (I \otimes \Delta^+)\Delta\;,\\
(e \otimes I)\Delta^+ &= (I \otimes e)\Delta^+ = I\;,\quad &
(\Delta^+ \otimes I)\Delta^+ &= (I \otimes \Delta^+)\Delta^+\;.
\end{equs}
Furthermore, there exists a map $\CA \colon T_+ \to T_+$ which is multiplicative in the sense that $\CA(\sigma \bar \sigma) = (\CA \sigma)\cdot(\CA \bar \sigma)$, and which is such that 
$\CM(I \otimes \CA)\Delta^+ = \CM(\CA \otimes I)\Delta^+ = e$, with $e\colon T_+ \to \R$ as in \eqref{e:defe}.
\end{lemma}

\begin{remark}
In technical lingo, this lemma states that $(T_+,\cdot,\Delta^+)$ is a Hopf algebra with antipode $\CA$,
and that $T$ is a comodule over $T_+$.
\end{remark}

The importance of this result is that it shows that $G$ is indeed a group. For any two morphisms $f$ and $g$,
we can define a linear map $f\circ g \colon T_+ \to \R$ by 
\begin{equ}
\bigl(f \circ g\bigr)(\sigma) = \bigl(f \otimes g\bigr)\Delta^+\sigma\;.
\end{equ}
As a consequence of the first identity in \eqref{e:idenDelta}, $f\circ g$ is again a morphism on $T_+$.
As a consequence of the second identity, one has $\hat \Gamma_{f \circ g} = \Gamma_f \,\Gamma_g$.
The last identity shows that $(f_1 \circ f_2) \circ f_3 = f_1 \circ (f_2 \circ f_3)$, while the properties
of $\CA$ ensure that if we set $f^{-1}(\sigma) = f(\CA \sigma)$, then $f \circ f^{-1} = f^{-1} \circ f = e$.
Finally, the third identity in \eqref{e:idenDelta} shows that $e$ is indeed the identity element, thus 
turning the set of all morphisms of $T_+$ into a group under $\circ$, acting on $T$ via $\hat \Gamma$.

Let us now turn back to our models. Given a smooth function $\xi_\eps$, we define $\Pi_z$ as in
\eqref{e:defPi} and $f_z$ by \eqref{e:deffz}. We then also define linear maps $\Gamma_{zz'}$ by 
$\Gamma_{zz'} = \hat \Gamma_{\gamma_{zz'}}$ with $\gamma_{zz'} = f_z^{-1} \circ f_{z'}$. We then have

\begin{lemma}\label{lem:lift}
For every smooth function $\xi_\eps$, the pair $(\Pi,\Gamma)$ defined above is a model.
\end{lemma}

\begin{proof}
The algebraic constraints \eqref{e:alg} are satisfied essentially by definition. The first bound
of \eqref{e:ana} can easily be verified recursively by \eqref{e:defPi}. The only non-trivial fact is that
the matrix elements of $\Gamma_{zz'}$ satisfy the right bound. If one can show that
$|\gamma_{zz'}(\sigma)| \lesssim |z-z'|^{|\sigma|}$, this in turn follows from \eqref{e:propDelta}.
This bound is non-trivial and was obtained in \cite[Prop.~8.27]{Regularity}.
\end{proof}

\subsection{Admissible models}

Thanks to Lemma~\ref{lem:lift}, we now have a large class of models for the regularity structures built in the
previous two subsections. However, we do not want to restrict ourselves to this class (or even its closure). The
reason is that if we define products in the ``na\"\i ve'' way given by the second identity in \eqref{e:defPi2}, then
there will typically be some situations where the result diverges as we let $\eps \to 0$ in $\xi_\eps$.
Therefore, we do not impose this relation in general but rather view it as the \textit{definition} of the product, i.e.\ we
interpret it as
\begin{equ}
\bigl(\Pi_z \tau\bigr)\cdot \bigl(\Pi_z \bar \tau\bigr) \eqdef \Pi_z(\tau \bar \tau)\;.
\end{equ}
However, the remainder of the structure described in \eqref{e:defPi} is required for $X_i$, $\DD_i$ and $\CI$ to
have the correct interpretation. This motivates the following definition.

\begin{definition}
Given a regularity structure $\TT$ constructed as in Sections~\ref{sec:spaces} and \ref{sec:group}, we say that a model
$(\Pi,\Gamma)$ is \textit{admissible} if it satisfies $\bigl(\Pi_z X_i\bigr)(z') = z_i' - z_i$,
$\Pi_z \DD_i\tau = \partial_i \Pi_z \tau$, as well as \eqref{e:defPi3} and if furthermore $\Gamma_{zz'} = \hat \Gamma_{f_z}^{-1}\hat \Gamma_{f_{z'}}$ with $f_z$ given by \eqref{e:deffz}. We will denote the space of all admissible models by $\MM_0 \subset \MM$. 
\end{definition}

\begin{remark}
Note that the definition of ``admissible'' requires us to fix a kernel $K$ which should be represented by $\CI$.
\end{remark}

\begin{remark}\label{rem:PPi}
In the particular case of admissible models for a regularity structure of the type considered here,
the data of the single linear map $\PPi$ as above is sufficient to reconstruct the full model $(\Pi,\Gamma)$.
\end{remark}

Note that at this stage, it is not clear whether this concept is even well-defined: 
in general, $D^{(k)}K \star \Pi_z \tau$ will be a distribution and cannot be evaluated at fixed points,
so \eqref{e:deffz} might be meaningless for a general model. It turns out that the definition actually
always makes sense, provided that the second identity in \eqref{e:deffz} is interpreted as
\begin{equ}
f_z(\CJ_k(\tau)) = \sum_{n \ge 0} \bigl(D^{(k)}K_n \star \Pi_z \tau\bigr)(z)\;,
\end{equ}
where $K = \sum_{n \ge 0} K_n$ as in \eqref{e:propKn}. This is because the bound \eqref{e:propKn}, combined
with the bound \eqref{e:ana} and the fact that $K_n$ is supported in the ball of radius $2^{-n}$ imply that
\begin{equ}
\bigl|\bigl(D^{(k)}K_n \star \Pi_z \tau\bigr)(z)\bigr| \lesssim 2^{(|k|-|\tau|-2)n}\;.
\end{equ}
The condition $|k| < |\tau| + 2$ appearing in \eqref{e:defPi3} is then precisely what is required to guarantee that this is always 
summable.

\subsection{Abstract fixed point problem}

We now show how to reformulate a stochastic PDE as a fixed point problem in some space $\CD^\gamma$ based
on an admissible model for the regularity structure associated to the SPDE by the construction of
Section~\ref{sec:spaces}. For definiteness, we focus on the example of the KPZ equation \eqref{e:KPZ},
but all other examples mentioned in the introduction can be treated in virtually the same way. Writing $P$ for the heat kernel,
the mild formulation of \eqref{e:KPZ} is given by
\begin{equ}[e:KPZint]
h = P \star \one_{t > 0}\bigl((\partial_x h)^2 + \xi\bigr) + Ph_0\;,
\end{equ}
where we write $Ph_0$ for the harmonic extension of $h_0$. (This is just the solution 
to the heat equation with initial condition $h_0$.) In order to formulate this as a fixed point 
problem in $\CD^\gamma$ for a suitable value of $\gamma > 0$, we will make use of the following
far-reaching extension of Schauder's theorem.

\begin{theorem}\label{thm:Schauder}
Fix one of the regularity structures built in the previous section and fix an admissible model. 
Then, for all but a discrete set of values of $\gamma > 0$, there exists a continuous operator
$\CP \colon \CD^\gamma \to \CD^{\gamma+2}$ such that the identity
\begin{equ}[e:propP]
\CR \CP f = P \star \CR f\;,
\end{equ}
holds for every $f \in \CD^\gamma$. Furthermore, one has $\bigl(\CP f\bigr)(z) - \CI f(z) \in \bar T$.
\end{theorem}

\begin{remark}
Recall that $\bar T \subset T$ denotes the linear span of the $X^k$, which represent the usual
Taylor polynomials. Again, while $\CP$ is a linear map when we consider the underlying model as fixed,
it can (and should) also be viewed as a continuous nonlinear map from $\MM_0 \ltimes \CD^\gamma$ into $\MM_0 \ltimes \CD^{\gamma+2}$.
The reason why some values of $\gamma$ need to be excluded is essentially the same as for the 
usual Schauder theorem. 
\end{remark}

For a proof of Theorem~\ref{thm:Schauder} and a precise description of the operator $\CP$, see \cite[Sec.~5]{Regularity}.
With the help of the operator $\CP$, it is then possible to reformulate \eqref{e:KPZint} as the following
fixed point problem in $\CD^\gamma$, provided that we have an admissible model at our disposal:
\begin{equ}[e:KPZFP]
H = \CP \one_{t > 0}\bigl((\DD H)^2 + \Xi\bigr) + Ph_0\;.
\end{equ}
Here, the smooth function $P h_0$ is interpreted as an element in $\CD^\gamma$ with values in $\bar T$
via its Taylor expansion of 
order $\gamma$. Note that in the context of the regularity structure associated to the KPZ equation
in Section~\ref{sec:spaces}, the right hand side of this equation makes sense for every 
$H\in \CD^\gamma$, provided that $H$ takes values
in $T_\CU$. This is an immediate consequence of the property \eqref{e:UVKPZ}. 

\begin{remark}
As already mentioned earlier, we cheat here in the sense that $\CD^\gamma$ should really be replaced by a
space $\CD^{\gamma,\eta}$ allowing for a suitable singular behaviour on the hyperplane $t = 0$.
\end{remark}

It is also possible to show (see \cite[Thm~4.7]{Regularity}) that if we set $|\Xi| = -{3\over 2} - \kappa$ for
some sufficiently small $\kappa > 0$, then one has $(\DD H)^2 \in \CD^{\gamma - {3\over 2} - \kappa}$ for
$H \in \CD^\gamma$. As a consequence, we expect to be able to find local solutions to the fixed point problem
\eqref{e:KPZFP}, provided that we formulate it in $\CD^\gamma$ for $\gamma > {3\over 2} + \kappa$.
This is indeed the case, and a more general instance of this fact can be found in \cite[Thm~7.8]{Regularity}.
Furthermore, the local solution is locally Lipschitz continuous as a function of both the initial condition
$h_0$ and the underlying admissible model $(\Pi,\Gamma) \in \MM_0$.
 
Now that we have a local solution $H \in \CD^\gamma$ for \eqref{e:KPZFP}, we would like to know how this 
solution relates to the original problem \eqref{e:KPZ}. This is given by the following simple fact:

\begin{proposition}\label{prop:PDE}
If the underlying model $(\Pi,\Gamma)$ is built from a smooth function $\xi_\eps$ as in \eqref{e:defPi}
and if $H$ solves \eqref{e:KPZFP}, then $\CR H$ solves \eqref{e:KPZint}.
\end{proposition}

\begin{proof}
As a consequence of \eqref{e:propP}, we see that $\CR H$ solves
\begin{equ}
\CR H = P \star \one_{t > 0}\bigl(\CR\bigl((\DD H)^2\bigr) + \xi_\eps\bigr) + Ph_0\;.
\end{equ}
Combining \eqref{e:defPi2} with \eqref{e:formulaRf}, it is not difficult to see that in this
particular case, one has $\CR\bigl((\DD H)^2\bigr) = (\partial_x \CR H)^2$, so that the claim follows.
\end{proof}

The results of the previous subsection yield a robust solution theory for \eqref{e:KPZFP}
which projects down (via $\CR$) to the usual solution theory for \eqref{e:KPZ} for smooth
driving noise $\xi_\eps$. If it were the case that the sequence of models $(\Pi^{(\eps)}, \Gamma^{(\eps)})$
associated to the regularised noise $\xi_\eps$ via \eqref{e:defPi} converges to a limit in $\MM_0$,
then this would essentially conclude our analysis of \eqref{e:KPZ}.

Unfortunately, this is \textit{not} the case. Indeed, in all of the examples mentioned in the introduction 
except for \eqref{e:Burgers}, the sequence of models $(\Pi^{(\eps)}, \Gamma^{(\eps)})$ does not converge
as $\eps \to 0$. In order to remedy to this situation, the idea is to look for a sequence of ``renormalised''
models $(\hat \Pi^{(\eps)}, \hat \Gamma^{(\eps)})$ which are also admissible and also satisfy
$\hat \Pi^{(\eps)}_z \Xi = \xi_\eps$, but do converge to a limit as $\eps \to 0$. The last section of this article
shows how these renormalised models can be constructed.

\subsection{Renormalisation}
\label{sec:renorm}

In order to renormalise our model, we will build a very natural group of continuous
transformations of $\MM_0$ that build a new admissible model from an old one. The renormalised model
will then be the image of the ``canonical'' model $(\Pi^{(\eps)}, \Gamma^{(\eps)})$ under a (diverging)
sequence of such transformations. 
Since we want the new model to also be admissible, the only defining property that we are allowed to
modify in \eqref{e:defPi} is the definition of the product. In order to describe the renormalised model, 
it turns out to be more convenient to consider again its representation by a single linear map
${\hat \PPi}^{(\eps)} \colon T \to \CS'$ as in \eqref{sec:group}, which is something we can do 
by Remark~\ref{rem:PPi}.

At this stage, we do not appear to have much choice: the only ``reasonable'' way of building 
$\hat \PPi^{(\eps)}$ from $\PPi^{(\eps)}$ is to compose it to the right with some fixed linear map $M_\eps \colon T \to T$:
\begin{equ}[e:renorm]
\hat \PPi^{(\eps)} = \PPi^{(\eps)} M_\eps\;.
\end{equ}
If we do this for an arbitrary map $M_\eps$, we will of course immediately lose the algebraic and analytical
properties that allow to associate an admissible model $(\hat \Pi^{(\eps)}, \hat \Gamma^{(\eps)})$ to the 
map $\hat \PPi^{(\eps)}$. As a matter of fact, it is completely unclear \textit{a priori} whether there exists \textit{any}
non-trivial map $M_\eps$ that preserves these properties. Fortunately, these maps do exists and a somewhat
indirect characterisation of them can be found in \cite[Sec.~8]{Regularity}. Even better, there are sufficiently
many of them so that the divergencies of $\PPi^{(\eps)}$ can be compensated by a judicious choice of $M_\eps$.

Let us just illustrate how this plays out in the case of the KPZ equation already studied in the
last subsection.
In order to simplify notations, we now use the following shorthand graphical 
notation for elements of $\CU \cup \CV$. For $\Xi$, we draw a small circle.
The integration map $\CI$ is then represented by a downfacing wavy line 
and $\DD\CI$ is represented by a downfacing plain line. The multiplication of 
symbols is obtained by joining them at the root. For example, we have
\begin{equ}
(\DD\CI(\Xi))^2 = \<2>\;,\qquad
(\DD\CI(\DD\CI(\Xi)^2))^2 = \<40>\;,\qquad
\CI(\DD\CI(\Xi)^2) = \<2d>\;.
\end{equ}
In the case of the KPZ equation, 
it turns out that one can exhibit an explicit four-parameter group of matrices $M$
which preserve admissible models when used in \eqref{e:renorm}.
These matrices are of the form $M = \exp(- \sum_{i=0}^3 C_i L_i)$, where the generators $L_i$ 
are determined by the following contraction rules:
\begin{equ}[e:defLL]
L_0 \colon \<11> \mapsto \one\;,\qquad L_1 \colon \<2> \mapsto \one\;,\qquad L_2 \colon \<40> \mapsto \one \qquad L_3 \colon \<211> \mapsto \one \;.
\end{equ}
This should be understood in the sense that if $\tau$ is an arbitrary formal expression,
then $L_0 \tau$ is the sum of all formal expressions obtained from $\tau$ by performing
a substitution of the type $\<11> \mapsto \one$. For example, one has
\begin{equ}
L_0 \<21> = 2 \<1>\;,\qquad L_0 \<211> = 2 \<11> + \<20>\;,
\end{equ}
etc. The extension of the other operators $L_i$ to all of $T$ proceeds in principle along the same 
lines. However, as a consequence of the fact that $\CI(\one) = 0$ by construction,
it actually turns out that $L_i \tau = 0$ for $i \neq 0$ and every $\tau$ for which 
$L_i$ wasn't already defined in \eqref{e:defLL}. We then have the following result, which is
a consequence of \cite[Sec.~8]{Regularity} and \cite{KPZJeremy} and was implicit in \cite{KPZ}:

\begin{theorem}
Let $M_\eps$ be given as above, let $\PPi^{(\eps)}$ be constructed from $\xi_\eps$ as in
\eqref{e:defPPi}, and let $\hat \PPi^{(\eps)} = \PPi^{(\eps)} M_\eps$. Then, there 
exists a unique admissible model $(\hat \Pi^{(\eps)}, \hat \Gamma^{(\eps)})$ such that 
$\hat \Pi_z^{(\eps)} = \hat \PPi^{(\eps)} \hat F_z^{(\eps)}$, where $\hat F_z^{(\eps)}$
relates to $\hat \Pi_z^{(\eps)}$ as in \eqref{e:deffz}. Furthermore, one has the identity
\begin{equ}[e:relrenorm]
\bigl(\hat \Pi_z^{(\eps)} \tau \bigr)(z) = \bigl(\Pi_z^{(\eps)} M_\eps \tau \bigr)(z)\;.
\end{equ}
Finally, there is a choice of $M_\eps$ such that $(\hat \Pi^{(\eps)}, \hat \Gamma^{(\eps)})$
converges to a limit $(\hat \Pi, \hat \Gamma)$ which is universal in that it does not depend on the
details of the regularisation procedure.
\end{theorem}

\begin{remark}
Despite \eqref{e:relrenorm}, it is \textit{not} true in general that $\hat \Pi_z^{(\eps)} = \Pi_z^{(\eps)}M_\eps$.
The point is that \eqref{e:relrenorm} only holds at the point $z$ and not at $z' \neq z$.
\end{remark}

In order to complete our survey of Theorem~\ref{theo:main}, it remains to identify the solution
to \eqref{e:KPZFP} with respect to the renormalised model $(\hat \Pi^{(\eps)}, \hat \Gamma^{(\eps)})$
with the classical solution to some modified partial differential equation. The continuity of the abstract
solution map then immediately implies that the solutions to the modified PDE converge to a limit.
The fact that the limiting model $(\hat \Pi, \hat \Gamma)$ is universal also implies that this limit is universal.

\begin{theorem}
Let $M_\eps = \exp(- \sum_{i=0}^3 C_i^{(\eps)} L_i)$ be as 
above and let $(\hat \Pi^{(\eps)}, \hat \Gamma^{(\eps)})$ be the corresponding renormalised model.
Let furthermore $H$ be the solution to \eqref{e:KPZFP} with respect to this model. Then, the function
$h(t,x) = \bigl(\CR H\bigr)(t,x)$ solves the equation
\begin{equ}[e:renormKPZ]
\partial_t h = \partial_x^2 h + (\partial_x h)^2 - 4C_0^{(\eps)}\, \partial_x h + \xi_\eps - (C_1^{(\eps)} + C_2^{(\eps)} + 4C_3^{(\eps)})\;.
\end{equ}
\end{theorem}

\begin{remark}
In order to obtain a limit $(\hat \Pi, \hat \Gamma)$, the renormalisation constants $C_i^{(\eps)}$ should
be chosen in the following way:
\begin{equ}
C_0^{(\eps)} = 0\;,\qquad C_1^{(\eps)} = {c_1 \over \eps}\;,\qquad C_2^{(\eps)} = 4 c_2 \log \eps + c_3\;,\qquad 
C_3^{(\eps)} = - c_2 \log \eps + c_4\;.
\end{equ}
Here, the $c_i$ are constants of order $1$ that are non-universal in the sense that they depend on
the details of the regularisation procedure for $\xi_\eps$. The fact that $C_0^{(\eps)} = 0$ explains why
the corresponding term does not appear in \eqref{e:KPZ}. The fact that the diverging parts of $C_2^{(\eps)}$ and $C_3^{(\eps)}$
cancel in \eqref{e:renormKPZ} explains why this logarithmic sub-divergence was not observed in \cite{MR1462228} for example.
\end{remark}

\begin{proof}
We first note that, as a consequence of Theorem~\ref{thm:Schauder} and of \eqref{e:KPZFP}, one can write
for $t > 0$
\begin{equ}[e:abstract]
H = \CI \bigl((\DD H)^2 + \Xi\bigr) + (...)\;,
\end{equ}
where $(...)$ denotes some terms belonging to $\bar T \subset T$. 

By repeatedly using this identity, we conclude that any solution $H \in \CD^\gamma$ to \eqref{e:KPZFP}
for $\gamma$ greater than (but close enough to) $3/2$ is necessarily of the form
\begin{equ}[e:expH]
H = h\, \one + \<d> + \<2d> + h'\,X_1 + 2 \<21d> + 2h'\, \<1d>\;,
\end{equ}
for some real-valued functions $h$ and $h'$. Note that $h'$ is treated
as an independent function here, we certainly do not mean to suggest that the function $h$ 
is differentiable! Our notation
is only by analogy with the classical Taylor expansion. As an immediate consequence, 
$\DD H$ is given by
\begin{equ}[e:DH]
\DD H = \<1> + \<20> + h'\,\one + 2 \<210> + 2h'\, \<10>\;,
\end{equ}
as an element of $\CD^\gamma$ for $\gamma$ close to $1/2$.
The right hand side of the equation is then given up to
order $0$ by
\begin{equ}
(\DD H)^2 + \Xi = \Xi + \<2> + 2\<21>+ 2h'\, \<1> + \<40> + 4\<211> + 2h'\, \<20> + 4h'\, \<11> + (h')^2\,\one\;. \label{e:RHS}
\end{equ}
Using the definition of $M_\eps$, we conclude that
\begin{equ}
M_\eps \DD H = \DD H - 4 C_0^{(\eps)} \<10>\;,
\end{equ}
so that, as an element of $\CD^\gamma$ with very small (but positive) $\gamma$,
one has the identity
\begin{equ}
(M_\eps \DD H)^2 = (\DD H)^2 - 8C_0^{(\eps)} \<11>\;.
\end{equ}
As a consequence, after neglecting again all terms of strictly positive homogeneity, one has the identity
\begin{equs}
M_\eps \bigl((\DD H)^2 + \Xi\bigr) &= (\DD H)^2 + \Xi - C_0^{(\eps)} \bigl(4\<1> + 4\<20> + 8\<11> + 4h'\,\one\bigr) - C_1^{(\eps)} -C_2^{(\eps)} -4C_3^{(\eps)} \\
&= (M_\eps\DD H)^2 + \Xi - 4 C_0^{(\eps)} \, M_\eps \DD H - (C_1^{(\eps)} +C_2^{(\eps)} +4C_3^{(\eps)})\;.
\end{equs}
Combining this with \eqref{e:relrenorm} and \eqref{e:formulaRf}, we conclude that
\begin{equ}
\CR \bigl((\DD H)^2 + \Xi\bigr) = (\partial_x \CR H)^2 + \xi_\eps - 4 C_0^{(\eps)} \, \partial_x \CR H - (C_1^{(\eps)} +C_2^{(\eps)} +4C_3^{(\eps)})\;,
\end{equ} 
from which the claim then follows in the same way as for Proposition~\ref{prop:PDE}.
\end{proof}

\begin{remark}\label{rem:appl}
Ultimately, the reason why the theory mentioned in Section~\ref{sec:Bony} (or indeed the theory of controlled 
rough paths, as originally exploited in \cite{KPZ}) can also be applied in this case 
is that in \eqref{e:expH}, only \textit{one} basis vector besides those in $\CT$ (i.e.\ besides $\one$ and $X_1$)
comes with a non-constant coefficient, namely the basis vector \<1d>. The methodology explained in
Section~\ref{sec:DPD} on the other hand can be applied whenever \textit{no} basis vector besides those in $\CT$ comes with
a non-constant coefficient.
\end{remark}

\bibliographystyle{./Martin}
\bibliography{./refs}

\begin{thebibliography}{DPDT07}
\expandafter\ifx\csname url\endcsname\relax
  \def\url#1{\texttt{#1}}\fi
\expandafter\ifx\csname urlprefix\endcsname\relax\def\urlprefix{URL }\fi

\bibitem[Aiz82]{MR678000}
\textsc{M.~Aizenman}.
\newblock Geometric analysis of {$\varphi ^{4}$} fields and {I}sing models.
  {I}, {II}.
\newblock \emph{Comm. Math. Phys.} \textbf{86}, no.~1, (1982), 1--48.

\bibitem[AR91]{AlbRock91}
\textsc{S.~Albeverio} and \textsc{M.~R{\"o}ckner}.
\newblock Stochastic differential equations in infinite dimensions: solutions
  via {D}irichlet forms.
\newblock \emph{Probab. Theory Related Fields} \textbf{89}, no.~3, (1991),
  347--386.
\newblock \ifx\href\undefined
  \texttt{doi:10.1007/BF01198791}\else\href{http://dx.doi.org/10.1007/BF01198791}{\texttt{doi:10.1007/BF01198791}}\fi.

\bibitem[BCD11]{BookChemin}
\textsc{H.~Bahouri}, \textsc{J.-Y. Chemin}, and \textsc{R.~Danchin}.
\newblock \emph{Fourier analysis and nonlinear partial differential equations},
  vol. 343 of \emph{Grundlehren der Mathematischen Wissenschaften}.
\newblock Springer, Heidelberg, 2011.

\bibitem[BG97]{MR1462228}
\textsc{L.~Bertini} and \textsc{G.~Giacomin}.
\newblock Stochastic {B}urgers and {KPZ} equations from particle systems.
\newblock \emph{Comm. Math. Phys.} \textbf{183}, no.~3, (1997), 571--607.
\newblock \ifx\href\undefined
  \texttt{doi:10.1007/s002200050044}\else\href{http://dx.doi.org/10.1007/s002200050044}{\texttt{doi:10.1007/s002200050044}}\fi.

\bibitem[Bis81a]{Bismut2}
\textsc{J.-M. Bismut}.
\newblock Martingales, the {M}alliavin calculus and {H}\"ormander's theorem.
\newblock In \emph{Stochastic integrals ({P}roc. {S}ympos., {U}niv. {D}urham,
  {D}urham, 1980)}, vol. 851 of \emph{Lecture Notes in Math.},  85--109.
  Springer, Berlin, 1981.

\bibitem[Bis81b]{Bismut1}
\textsc{J.-M. Bismut}.
\newblock Martingales, the {M}alliavin calculus and hypoellipticity under
  general {H}\"ormander's conditions.
\newblock \emph{Z. Wahrsch. Verw. Gebiete} \textbf{56}, no.~4, (1981),
  469--505.
\newblock \ifx\href\undefined
  \texttt{doi:10.1007/BF00531428}\else\href{http://dx.doi.org/10.1007/BF00531428}{\texttt{doi:10.1007/BF00531428}}\fi.

\bibitem[BMN10]{MR2682821}
\textsc{{\'A}.~B{\'e}nyi}, \textsc{D.~Maldonado}, and \textsc{V.~Naibo}.
\newblock What is {$\ldots$} a paraproduct?
\newblock \emph{Notices Amer. Math. Soc.} \textbf{57}, no.~7, (2010), 858--860.

\bibitem[Bon81]{Bony}
\textsc{J.-M. Bony}.
\newblock Calcul symbolique et propagation des singularit\'es pour les
  \'equations aux d\'eriv\'ees partielles non lin\'eaires.
\newblock \emph{Ann. Sci. \'Ecole Norm. Sup. (4)} \textbf{14}, no.~2, (1981),
  209--246.

\bibitem[BPRS93]{MR1317994}
\textsc{L.~Bertini}, \textsc{E.~Presutti}, \textsc{B.~R{\"u}diger}, and
  \textsc{E.~Saada}.
\newblock Dynamical fluctuations at the critical point: convergence to a
  nonlinear stochastic {PDE}.
\newblock \emph{Teor. Veroyatnost. i Primenen.} \textbf{38}, no.~4, (1993),
  689--741.

\bibitem[CC13]{Chouk}
\textsc{R.~{Catellier}} and \textsc{K.~{Chouk}}.
\newblock Paracontrolled distributions and the 3-dimensional stochastic
  quantization equation.
\newblock \emph{ArXiv e-prints} (2013).
\newblock \ifx\href\undefined
  \texttt{arXiv:1310.6869}\else\href{http://arxiv.org/abs/1310.6869}{\texttt{arXiv:1310.6869}}\fi.

\bibitem[CHLT12]{Tom}
\textsc{T.~{Cass}}, \textsc{M.~{Hairer}}, \textsc{C.~{Litterer}}, and
  \textsc{S.~{Tindel}}.
\newblock Smoothness of the density for solutions to {G}aussian rough
  differential equations.
\newblock \emph{ArXiv e-prints} (2012).
\newblock \ifx\href\undefined
  \texttt{arXiv:1209.3100}\else\href{http://arxiv.org/abs/1209.3100}{\texttt{arXiv:1209.3100}}\fi.
\newblock Ann. Probab., to appear.

\bibitem[CM94]{MR1185878}
\textsc{R.~A. Carmona} and \textsc{S.~A. Molchanov}.
\newblock Parabolic {A}nderson problem and intermittency.
\newblock \emph{Mem. Amer. Math. Soc.} \textbf{108}, no. 518, (1994), viii+125.

\bibitem[Dob79]{Dobrushin}
\textsc{R.~L. Dobrushin}.
\newblock Gaussian and their subordinated self-similar random generalized
  fields.
\newblock \emph{Ann. Probab.} \textbf{7}, no.~1, (1979), 1--28.
\newblock \ifx\href\undefined
  \texttt{doi:10.1214/aop/1176995145‎}\else\href{http://dx.doi.org/10.1214/aop/1176995145‎}{\texttt{doi:10.1214/aop/1176995145‎}}\fi.

\bibitem[DPD02]{MR1941997}
\textsc{G.~Da~Prato} and \textsc{A.~Debussche}.
\newblock Two-dimensional {N}avier-{S}tokes equations driven by a space-time
  white noise.
\newblock \emph{J. Funct. Anal.} \textbf{196}, no.~1, (2002), 180--210.
\newblock \ifx\href\undefined
  \texttt{doi:10.1006/jfan.2002.3919}\else\href{http://dx.doi.org/10.1006/jfan.2002.3919}{\texttt{doi:10.1006/jfan.2002.3919}}\fi.

\bibitem[DPD03]{MR2016604}
\textsc{G.~Da~Prato} and \textsc{A.~Debussche}.
\newblock Strong solutions to the stochastic quantization equations.
\newblock \emph{Ann. Probab.} \textbf{31}, no.~4, (2003), 1900--1916.
\newblock \ifx\href\undefined
  \texttt{doi:10.1214/aop/1068646370}\else\href{http://dx.doi.org/10.1214/aop/1068646370}{\texttt{doi:10.1214/aop/1068646370}}\fi.

\bibitem[DPDT07]{MR2365646}
\textsc{G.~Da~Prato}, \textsc{A.~Debussche}, and \textsc{L.~Tubaro}.
\newblock A modified {K}ardar-{P}arisi-{Z}hang model.
\newblock \emph{Electron. Comm. Probab.} \textbf{12}, (2007), 442--453
  (electronic).
\newblock \ifx\href\undefined
  \texttt{doi:10.1214/ECP.v12-1333}\else\href{http://dx.doi.org/10.1214/ECP.v12-1333}{\texttt{doi:10.1214/ECP.v12-1333}}\fi.

\bibitem[DPZ92]{DPZ}
\textsc{G.~Da~Prato} and \textsc{J.~Zabczyk}.
\newblock \emph{Stochastic Equations in Infinite Dimensions}, vol.~44 of
  \emph{Encyclopedia of Mathematics and its Applications}.
\newblock Cambridge University Press, 1992.

\bibitem[FH14]{Peter}
\textsc{P.~K. Friz} and \textsc{M.~Hairer}.
\newblock \emph{A course on rough paths}.
\newblock Universitext. Springer, 2014.
\newblock To appear.

\bibitem[Fr{\"o}82]{MR643591}
\textsc{J.~Fr{\"o}hlich}.
\newblock On the triviality of {$\lambda \varphi ^{4}_{d}$} theories and the
  approach to the critical point in {$d>4$} dimensions.
\newblock \emph{Nuclear Phys. B} \textbf{200}, no.~2, (1982), 281--296.
\newblock \ifx\href\undefined
  \texttt{doi:10.1016/0550-3213(82)90088-8}\else\href{http://dx.doi.org/10.1016/0550-3213(82)90088-8}{\texttt{doi:10.1016/0550-3213(82)90088-8}}\fi.

\bibitem[FV10]{MR2604669}
\textsc{P.~K. Friz} and \textsc{N.~B. Victoir}.
\newblock \emph{Multidimensional stochastic processes as rough paths}, vol. 120
  of \emph{Cambridge Studies in Advanced Mathematics}.
\newblock Cambridge University Press, Cambridge, 2010.
\newblock Theory and applications.

\bibitem[GIP12]{PAMPreprint}
\textsc{M.~{Gubinelli}}, \textsc{P.~{Imkeller}}, and \textsc{N.~{Perkowski}}.
\newblock Paraproducts, rough paths and controlled distributions.
\newblock \emph{ArXiv e-prints} (2012).
\newblock \ifx\href\undefined
  \texttt{arXiv:1210.2684}\else\href{http://arxiv.org/abs/1210.2684}{\texttt{arXiv:1210.2684}}\fi.

\bibitem[GJ10]{Milton}
\textsc{P.~{Goncalves}} and \textsc{M.~{Jara}}.
\newblock {Universality of KPZ equation}.
\newblock \emph{ArXiv e-prints} (2010).
\newblock \ifx\href\undefined
  \texttt{arXiv:1003.4478}\else\href{http://arxiv.org/abs/1003.4478}{\texttt{arXiv:1003.4478}}\fi.

\bibitem[Gub04]{Max}
\textsc{M.~Gubinelli}.
\newblock Controlling rough paths.
\newblock \emph{J. Funct. Anal.} \textbf{216}, no.~1, (2004), 86--140.
\newblock \ifx\href\undefined
  \texttt{doi:10.1016/j.jfa.2004.01.002}\else\href{http://dx.doi.org/10.1016/j.jfa.2004.01.002}{\texttt{doi:10.1016/j.jfa.2004.01.002}}\fi.

\bibitem[Hai11]{Rough}
\textsc{M.~Hairer}.
\newblock Rough stochastic {PDE}s.
\newblock \emph{Comm. Pure Appl. Math.} \textbf{64}, no.~11, (2011),
  1547--1585.
\newblock \ifx\href\undefined
  \texttt{doi:10.1002/cpa.20383}\else\href{http://dx.doi.org/10.1002/cpa.20383}{\texttt{doi:10.1002/cpa.20383}}\fi.

\bibitem[{Hai}13]{KPZ}
\textsc{M.~{Hairer}}.
\newblock {Solving the KPZ equation}.
\newblock \emph{Ann. of Math. (2)} \textbf{178}, no.~2, (2013), 559--664.
\newblock \ifx\href\undefined
  \texttt{doi:10.4007/annals.2013.178.2.4}\else\href{http://dx.doi.org/10.4007/annals.2013.178.2.4}{\texttt{doi:10.4007/annals.2013.178.2.4}}\fi.

\bibitem[Hai14a]{Notes}
\textsc{M.~Hairer}.
\newblock Introduction to regularity structures.
\newblock \emph{ArXiv e-prints} (2014).
\newblock \ifx\href\undefined
  \texttt{arXiv:1401.3014}\else\href{http://arxiv.org/abs/1401.3014}{\texttt{arXiv:1401.3014}}\fi.
\newblock Braz. J. Prob. Stat., to appear.

\bibitem[Hai14b]{Regularity}
\textsc{M.~Hairer}.
\newblock A theory of regularity structures.
\newblock \emph{Invent. Math.} (2014).
\newblock \ifx\href\undefined
  \texttt{doi:10.1007/s00222-014-0505-4}\else\href{http://dx.doi.org/10.1007/s00222-014-0505-4}{\texttt{doi:10.1007/s00222-014-0505-4}}\fi.

\bibitem[HM11]{ErgodicBig}
\textsc{M.~Hairer} and \textsc{J.~C. Mattingly}.
\newblock A theory of hypoellipticity and unique ergodicity for semilinear
  stochastic {PDE}s.
\newblock \emph{Electron. J. Probab.} \textbf{16}, (2011), 658--738.
\newblock \ifx\href\undefined
  \texttt{doi:10.1214/EJP.v16-875}\else\href{http://dx.doi.org/10.1214/EJP.v16-875}{\texttt{doi:10.1214/EJP.v16-875}}\fi.

\bibitem[HM12]{Jan}
\textsc{M.~Hairer} and \textsc{J.~Maas}.
\newblock A spatial version of the {I}t\^o-{S}tratonovich correction.
\newblock \emph{Ann. Probab.} \textbf{40}, no.~4, (2012), 1675--1714.
\newblock \ifx\href\undefined
  \texttt{doi:10.1214/11-AOP662}\else\href{http://dx.doi.org/10.1214/11-AOP662}{\texttt{doi:10.1214/11-AOP662}}\fi.

\bibitem[HMW14]{JanHendrik}
\textsc{M.~Hairer}, \textsc{J.~Maas}, and \textsc{H.~Weber}.
\newblock Approximating rough stochastic {PDE}s.
\newblock \emph{Comm. Pure Appl. Math.} \textbf{67}, no.~5, (2014), 776--870.
\newblock \ifx\href\undefined
  \texttt{doi:10.1002/cpa.21495}\else\href{http://dx.doi.org/10.1002/cpa.21495}{\texttt{doi:10.1002/cpa.21495}}\fi.

\bibitem[HP13]{Natesh}
\textsc{M.~Hairer} and \textsc{N.~S. Pillai}.
\newblock Regularity of laws and ergodicity of hypoelliptic {SDE}s driven by
  rough paths.
\newblock \emph{Ann. Probab.} \textbf{41}, no.~4, (2013), 2544--2598.
\newblock \ifx\href\undefined
  \texttt{doi:10.1214/12-AOP777}\else\href{http://dx.doi.org/10.1214/12-AOP777}{\texttt{doi:10.1214/12-AOP777}}\fi.

\bibitem[HPP14]{HPP}
\textsc{M.~Hairer}, \textsc{{\'E}.~Pardoux}, and \textsc{A.~Piatnitksy}.
\newblock A {W}ong-{Z}akai theorem for stochastic {PDE}s, 2014.
\newblock Work in progress.

\bibitem[HQ14]{KPZJeremy}
\textsc{M.~Hairer} and \textsc{J.~Quastel}.
\newblock Continuous interface models rescale to {KPZ}, 2014.
\newblock Work in progress.

\bibitem[HW13]{Hendrik}
\textsc{M.~Hairer} and \textsc{H.~Weber}.
\newblock Rough {B}urgers-like equations with multiplicative noise.
\newblock \emph{Probab. Theory Related Fields} \textbf{155}, no. 1-2, (2013),
  71--126.
\newblock \ifx\href\undefined
  \texttt{doi:10.1007/s00440-011-0392-1}\else\href{http://dx.doi.org/10.1007/s00440-011-0392-1}{\texttt{doi:10.1007/s00440-011-0392-1}}\fi.

\bibitem[It{\^o}44]{Ito}
\textsc{K.~It{\^o}}.
\newblock Stochastic integral.
\newblock \emph{Proc. Imp. Acad. Tokyo} \textbf{20}, (1944), 519--524.

\bibitem[JLM85]{MR815192}
\textsc{G.~Jona-Lasinio} and \textsc{P.~K. Mitter}.
\newblock On the stochastic quantization of field theory.
\newblock \emph{Comm. Math. Phys.} \textbf{101}, no.~3, (1985), 409--436.

\bibitem[KPZ86]{KPZOrig}
\textsc{M.~Kardar}, \textsc{G.~Parisi}, and \textsc{Y.-C. Zhang}.
\newblock Dynamic scaling of growing interfaces.
\newblock \emph{Phys. Rev. Lett.} \textbf{56}, no.~9, (1986), 889--892.

\bibitem[KS84]{KSAMI}
\textsc{S.~Kusuoka} and \textsc{D.~Stroock}.
\newblock Applications of the {M}alliavin calculus. {I}.
\newblock In \emph{Stochastic analysis (Katata/Kyoto, 1982)}, vol.~32 of
  \emph{North-Holland Math. Library},  271--306. North-Holland, Amsterdam,
  1984.

\bibitem[KS85]{KSAMII}
\textsc{S.~Kusuoka} and \textsc{D.~Stroock}.
\newblock Applications of the {M}alliavin calculus. {II}.
\newblock \emph{J. Fac. Sci. Univ. Tokyo Sect. IA Math.} \textbf{32}, no.~1,
  (1985), 1--76.

\bibitem[KS87]{KSAMIII}
\textsc{S.~Kusuoka} and \textsc{D.~Stroock}.
\newblock Applications of the {M}alliavin calculus. {III}.
\newblock \emph{J. Fac. Sci. Univ. Tokyo Sect. IA Math.} \textbf{34}, no.~2,
  (1987), 391--442.

\bibitem[LCL07]{MR2314753}
\textsc{T.~J. Lyons}, \textsc{M.~Caruana}, and \textsc{T.~L{\'e}vy}.
\newblock \emph{Differential equations driven by rough paths}, vol. 1908 of
  \emph{Lecture Notes in Mathematics}.
\newblock Springer, Berlin, 2007.
\newblock Lectures from the 34th Summer School on Probability Theory held in
  Saint-Flour.

\bibitem[LQ02]{MR2036784}
\textsc{T.~J. Lyons} and \textsc{Z.~Qian}.
\newblock \emph{System control and rough paths}.
\newblock Oxford Mathematical Monographs. Oxford University Press, Oxford,
  2002.
\newblock Oxford Science Publications.

\bibitem[Lyo98]{Lyons}
\textsc{T.~J. Lyons}.
\newblock Differential equations driven by rough signals.
\newblock \emph{Rev. Mat. Iberoamericana} \textbf{14}, no.~2, (1998), 215--310.

\bibitem[Mal97]{Malliavin}
\textsc{P.~Malliavin}.
\newblock \emph{Stochastic analysis}, vol. 313 of \emph{Grundlehren der
  Mathematischen Wissenschaften}.
\newblock Springer-Verlag, Berlin, 1997.

\bibitem[Nor86]{Nor86SMC}
\textsc{J.~Norris}.
\newblock Simplified {M}alliavin calculus.
\newblock In \emph{S\'eminaire de Probabilit\'es, XX, 1984/85}, vol. 1204 of
  \emph{Lecture Notes in Math.},  101--130. Springer, Berlin, 1986.

\bibitem[Nua06]{Nualart}
\textsc{D.~Nualart}.
\newblock \emph{The {M}alliavin calculus and related topics}.
\newblock Probability and its Applications (New York). Springer-Verlag, Berlin,
  second ed., 2006.

\bibitem[PW81]{ParisiWu}
\textsc{G.~Parisi} and \textsc{Y.~S. Wu}.
\newblock Perturbation theory without gauge fixing.
\newblock \emph{Sci. Sinica} \textbf{24}, no.~4, (1981), 483--496.

\bibitem[Sch34]{Schauder}
\textsc{J.~Schauder}.
\newblock \"{U}ber lineare elliptische {D}ifferentialgleichungen zweiter
  {O}rdnung.
\newblock \emph{Math. Z.} \textbf{38}, no.~1, (1934), 257--282.
\newblock \ifx\href\undefined
  \texttt{doi:10.1007/BF01170635}\else\href{http://dx.doi.org/10.1007/BF01170635}{\texttt{doi:10.1007/BF01170635}}\fi.

\bibitem[Sim97]{MR1459795}
\textsc{L.~Simon}.
\newblock Schauder estimates by scaling.
\newblock \emph{Calc. Var. Partial Differential Equations} \textbf{5}, no.~5,
  (1997), 391--407.
\newblock \ifx\href\undefined
  \texttt{doi:10.1007/s005260050072}\else\href{http://dx.doi.org/10.1007/s005260050072}{\texttt{doi:10.1007/s005260050072}}\fi.

\bibitem[Str64]{Strat}
\textsc{R.~L. Stratonovi{\v{c}}}.
\newblock A new form of representing stochastic integrals and equations.
\newblock \emph{Vestnik Moskov. Univ. Ser. I Mat. Meh.} \textbf{1964}, no.~1,
  (1964), 3--12.

\bibitem[WZ65a]{WongZakai2}
\textsc{E.~Wong} and \textsc{M.~Zakai}.
\newblock On the convergence of ordinary integrals to stochastic integrals.
\newblock \emph{Ann. Math. Statist.} \textbf{36}, (1965), 1560--1564.
\newblock \ifx\href\undefined
  \texttt{doi:10.1214/aoms/1177699916}\else\href{http://dx.doi.org/10.1214/aoms/1177699916}{\texttt{doi:10.1214/aoms/1177699916}}\fi.

\bibitem[WZ65b]{WongZakai}
\textsc{E.~Wong} and \textsc{M.~Zakai}.
\newblock On the relation between ordinary and stochastic differential
  equations.
\newblock \emph{Internat. J. Engrg. Sci.} \textbf{3}, (1965), 213--229.

\end{thebibliography}

\end{document}